\documentclass{article}
\usepackage{amsthm}
\usepackage{amsmath}
\usepackage{amssymb}
\usepackage{mathtools}
\usepackage{enumitem}
\usepackage{geometry}
\usepackage{graphicx}
\usepackage{xcolor}
\usepackage[hidelinks]{hyperref}
\usepackage{cleveref}

\allowdisplaybreaks

\newtheorem{thm}{Theorem}
\newtheorem{prop}[thm]{Proposition}
\newtheorem{rem}[thm]{Remark}

\newtheorem{lem}[thm]{Lemma}
\newtheorem{cor}[thm]{Corollary}

\newtheorem{assn}[thm]{Assumption}

\crefname{thm}{theorem}{theorems}
\crefname{defn}{definition}{definitions}
\crefname{rem}{remark}{remarks}
\crefname{prop}{proposition}{propositions}
\crefname{lem}{lemma}{lemmas}
\crefname{section}{section}{sections}

\usepackage{algpseudocode}
\usepackage{setspace} 
\usepackage{float}
\usepackage{subcaption}

\newfloat{algorithm}{h}{loa}
\floatname{algorithm}{Algorithm}

\newcommand{\R}{\mathbb{R}}

\newcommand{\Z}{\mathbb{Z}}

\newcommand{\C}{\mathbb{C}}

\newcommand{\mres}{\mathbin{\vrule height 1.6ex depth 0pt width
0.13ex\vrule height 0.13ex depth 0pt width 1.3ex}}

\newcommand{\dual}{*}

\DeclareMathOperator{\TV}{TV}
\DeclareMathOperator{\BV}{BV}
\newcommand{\D}{\mathrm{D}}
\newcommand{\TVp}{\TV}

\newcommand{\TVd}[1]{\TV_{#1}}

\newcommand{\Dvalid}{D_{\text{valid}}}
\newcommand{\Dinvalid}{D_{\text{invalid}}}
\newcommand{\Dconv}{D_{\text{conv}}}

\newcommand{\BregmanDistance}[2]{\mathcal{D}^{#1}_{#2}}

\newcommand{\prodm}[1]{ {\otimes}_{#1}}

\newcommand{\av}{{\mathrm{av}}}
\newcommand{\ver}{{\mathrm v}}
\newcommand{\hor}{{\mathrm h}}
\newcommand{\Lv}{L^\ver}
\newcommand{\Lh}{L^\hor}
\newcommand{\Lvm}{L^{\ver,m}}
\newcommand{\Lhn}{L^{\hor,n}}

\newcommand{\fcut}{\Phi}
\newcommand{\Dir}{\mathfrak D}
\newcommand{\Fej}{\mathfrak F_{\fcut}}
\newcommand{\cha}{\chi^{\fcut}}

\newcommand{\hd}{\mathcal H} 

\newcommand{\wrt}{\:\mathrm{d}}

\DeclareMathOperator{\sgn}{sgn}

\newcommand{\e}{{\mathrm e}}
\renewcommand{\d}{{\mathrm d}}

\newcommand{\sinc}{\mathrm{sinc}}
\newcommand{\Si}{\mathrm{Si}}
\newcommand{\Ci}{\mathrm{Ci}}
\newcommand{\Cin}{\mathrm{Cin}}

\newcommand{\Sc}{ \mathcal{S} }

\newcommand{\notinclude}[1]{}

\newcommand{\todo}[1]{{\color{red}[TODO: {#1}]}}

\DeclareMathOperator{\supp}{supp}
\DeclareMathOperator{\range}{Rg}
\DeclareMathOperator{\dive}{div}
\DeclareMathOperator{\subdif}{\partial}

\DeclareMathOperator{\dist}{dist}

\newcommand{\st}{ \, \left|\right.\, }
\DeclarePairedDelimiterX\set[1]\lbrace\rbrace{#1}
\DeclarePairedDelimiterX{\abs}[1]{|}{|}{#1}
\DeclareMathOperator{\length}{length}

\newlist{assumptions}{enumerate}{1}
\setlist[assumptions]{label=\arabic*)}
\crefname{assumptionsi}{assumption}{assumptions}

\title{Exact reconstruction and reconstruction from noisy data with anisotropic total variation}
\author{M. Holler \and  B. Wirth}
\date{}

\begin{document}
\maketitle
\begin{abstract}
It is well-known that point sources with sufficient mutual distance can be reconstructed exactly from finitely many Fourier measurements by solving a convex optimization problem with Tikhonov-regularization (this property is sometimes termed superresolution). In case of noisy measurements one can bound the reconstruction error in unbalanced Wasserstein distances or weak Sobolev-type norms.
A natural question is to what other settings the phenomenon of superresolution extends. We here keep the same measurement operator, but replace the regularizer to anisotropic total variation, which is particularly suitable for regularizing piecewise constant images with horizontal and vertical edges. Under sufficient mutual distance between the horizontal and vertical edges we prove exact reconstruction results and $L^1$ error bounds in terms of the measurement noise.
\end{abstract}

\tableofcontents

\section{Introduction}

Consider the inverse problem to reconstruct a function or a measure $u^\dagger$ on the flat torus $\Omega = (\R/\Z)^2$
from a measurement $f^\dagger=Ku^\dagger$ consisting of finitely many Fourier coefficients,
\begin{equation}\label{eqn:truncatedFourier}
Ku
=(\hat u_k)_{k\in\Z^2,|k|_\infty\leq\fcut}
=\left(\int_\Omega e^{-2\pi i(x,y)\cdot(k_1,k_2)}\wrt u(x,y)\right)_{k_1,k_2\in\Z,\,|k_1|,|k_2|\leq\fcut}.
\end{equation}
If $u^\dagger=\sum_{i=1}^Ia_i\delta_{(x_i,y_i)}$ is a linear combination of Dirac masses at locations $(x_i,y_i)\in\Omega$
and if those locations have sufficient mutual distance in terms of the cutoff frequency $\fcut\geq1$,
then $u^\dagger$ can be uniquely reconstructed as the minimizer of
\begin{equation*}
J_0^{f^\dagger}(u)=|u|(\Omega)+\iota_{\{f^\dagger\}}(Ku),
\end{equation*}
where $\iota_A$ denotes the convex indicator function of the set $A$ (see for instance \cite{Castro2012,Candes2014,AlbertiAmmariRomeroWintz2019,HollerSchlueterWirth2022}).
In case of noisy measurements $f^\delta$ with noise strength $\frac12|f^\delta-f^\dagger|_2^2\leq\delta$,
the ground truth $u^\dagger$ can be approximated by the minimizer $u^\delta$ of
\begin{equation*}
J_\alpha^{f^\delta}(u)=|u|(\Omega)+\tfrac1{2\alpha}|Ku-f^\delta|_2^2,
\end{equation*}
where for the choice $\alpha=\sqrt\delta$ one obtains a reconstruction error of order $\sqrt\delta$ in the unbalanced Wasserstein-2 distance or in RKHS norms
(see for instance \cite{Candes2013,Duval2015,HollerSchlueterWirth2022}).
A natural question is to what other inverse problems settings such results can be extended.
One direction concerns the measurement operator $K$, for which one might for instance consider discretized Laplace or X-ray transforms or the like (see for instance \cite{DenoyelleDuvalPeyreSoubies2020}).
In this article we deal with another direction, namely changing the regularization.
More specifically, we replace the total mass $|u|(\Omega)$ in the above minimization problems by the anisotropic total variation
\begin{equation*}
\TVp(u) = \sup \left \{ \int_\Omega u \dive \varphi\wrt(x,y) \,\middle|\, \varphi \in C_c^\infty(\Omega,\R^2),\, \sup_{(x,y)\in\Omega}|\varphi(x,y)|_\infty\leq 1 \right\}
\end{equation*}
(with periodic boundary conditions; homogeneous Dirichlet conditions could be used as well).
This anisotropic total variation penalizes the (distributional) gradient of an image $u:\Omega\to\R$
and is known to lead to sparse gradients or equivalently sharp image edges with a preference for horizontal and vertical edges
(the cost provides no incentive for cutting off corners by diagonal edges).
Since images with just horizontal and vertical edges can be parameterized by finitely many values
one can expect a similar behaviour as for the reconstruction of point masses, which is what we show.

\subsection{Setting and main results}
In detail, we consider piecewise constant real-valued ground truth images
\begin{equation*}
u^\dagger = \sum_{m=1}^M \sum_{n=1}^N u^\dagger_{mn} \chi_{[x_m,x_{m+1}[ \times [y_n,y_{n+1}[}\in\BV(\Omega)
\end{equation*}
in the space of functions of bounded variation (with periodic boundary conditions),
where $\chi_A$ denotes the characteristic function of a set $A$ and $x_1<\ldots<x_M$ as well as $y_1<\ldots<y_N$ are points in $[0,1[$. %
Here and throughout we identify $x_{M+1}=x_1$ as well as $y_{N+1}=y_1$ and interpret an interval $[a,b[$ with $b<a$ as $[a,1+b[\subset\R$ projected onto $\R/\Z$,
the latter being identified with $[0,1[$ in the canonical way.
Denoting by $\dist$ the metric on $\R/\Z$,
the minimum distance of the points $x_1, \ldots, x_M, y_1,\ldots,y_N$ regarded as points in $\R/\Z$ will be denoted
\begin{equation*}
\Delta=\min\{\min\{\dist(x_m,x_{m+1})\,|\,m=1,\ldots,M\},\min\{\dist(y_n,y_{n+1})\,|\,n=1,\ldots,N\}\}.
\end{equation*}
We further assume the following.

\begin{assn}[Consistent gradient direction]\label{ass:consistentGradient}
On a vertical line, the (distributional) $x$-derivative of $u^\dagger$ does not change sign,
and likewise, on a horizontal line its $y$-derivative does not change sign.
In other words, for each $m$ the restriction $\D_xu^\dagger\mres(\{x_m\}\times\R/\Z)$ of the measure $\D_xu^\dagger$ to the vertical line $\{x_m\}\times\R/\Z$
is either a nonnegative or a nonpositive measure, as is $\D_yu^\dagger\mres(\R/\Z\times\{y_n\})$ for each $n$.
\end{assn}

Then one indeed obtains exact reconstruction from noise-free measurements $f^\dagger=Ku^\dagger$.

\begin{thm}[Exact reconstruction] \label{thm:exact_recon_intro}
Let \cref{ass:consistentGradient} hold and $K$ be the operator \eqref{eqn:truncatedFourier}.
There exists a constant $C>0$ such that, if $\Delta>\frac C{\fcut}$, then $u^\dagger$ is the unique minimizer of
\begin{equation*}
J_0^{f^\dagger}(u)=\TVp(u)+\iota_{\{f^\dagger\}}(Ku).
\end{equation*}
\end{thm}

In case of noisy measurements $f^\delta$ with $\frac12|f^\delta-f^\dagger|_2^2<\delta$ one obtains a bound on the reconstruction error in the $L^1(\Omega)$-norm $\|\cdot\|_1$.

\begin{thm}[Convergence for vanishing noise] \label{thm:convergence_intro}
Let \cref{ass:consistentGradient} hold and $K$ be the operator \eqref{eqn:truncatedFourier}.
There exists a constant $C>0$ such that, if $\Delta>\frac C{\fcut}$, then any minimizer $u^\delta$ of
\begin{equation*}
J_{\alpha}^{f^\delta}(u)=\TVp(u)+\tfrac1{2\alpha}|Ku-f^\delta|_2^2
\end{equation*}
for the choice $\alpha=\sqrt\delta$ satisfies
\begin{equation*}
\|u^\delta-u^\dagger\|_1\leq C\delta^{1/4}.
\end{equation*}
\end{thm}

Both theorems will be proved in \cref{sec:reconstructionError} under source conditions on $u^\dagger$,
while we prove in \cref{sec:dualVariables} that these source conditions are satisfied for $K$ being the truncated Fourier series and $\Delta$ large enough.
The latter proofs will be based on the explicit construction of dual variables via trigonometric polynomials,
the simplest of which was essentially already given in \cite{Candes2014}.

\begin{rem}[Optimality of results]\label{rem:rate_not_optimal} Even though \cref{ass:consistentGradient} appears rather technical, we still expect a (possibly weakened) assumption of this type to be necessary for exact reconstruction; evidence for this will be provided by our numerical experiments. The rate $\delta^{1/4}$ of theorem \ref{thm:convergence_intro} on the other hand is expected to be suboptimal. Indeed, in the case of sparse spikes superresolution, results corresponding to a rate of $\delta^{1/2}$ were provided in \cite{Duval2015}, and we conjecture that
an improved convergence rate also holds in our setting. This conjecture is backed up by our numerical experiments, which clearly show the improved rate of $\delta^{1/2}$. 
\end{rem}

\subsection{Notation and preliminaries}\label{sec:notation}
As already done before, we will denote the $\ell^p$-norm on a finite-dimensional vector space by $|\cdot|_p$, while the $L^p(\Omega)$-norm is denoted $\|\cdot\|_p$.
We write $\dist((x,y),(\tilde x,\tilde y))=\min\{|(x-\tilde x+k,y-\tilde y+l)|_2\,|\,k,l\in\Z\}$ for the Euclidean distance between two points $(x,y),(\tilde x,\tilde y)\in\Omega$.
For $R>0$ the open Euclidean $R$-neighbourhood of a set $A\subset\Omega$ is denoted $B_R(A)=\{(x,y)\in\Omega\,|\,\dist((x,y),A)<R\}$,
where $\dist((x,y),A)$ is short for $\inf\{\dist((x,y),(\tilde x,\tilde y))\,|\,(\tilde x,\tilde y)\in A\}$.
The analogous notation is used on $\R/\Z$.
Furthermore, by $A^c$ we denote the complement of $A$ in $\Omega$ or $\R/\Z$.
For simplicity we will often identify $\R/\Z$ and $[0,1]$ with periodic boundary conditions (thus $0$ is identified with $1$).
Consistently, for $0\leq a\leq b\leq 1$ we denote by $[a,b]$ the usual interval, while for $0\leq b\leq a\leq1$ we interpret $[a,b]$ as $[0,b]\cup[a,1]$ (analogous for open intervals).
Also, contrary to the usual convention, we will write $\int_a^b$ for $\int_{[a,b]}$.
The length of such an interval will be denoted $\length([a,b])$.
For $r,s>0$ we will use the notation $r\lesssim s$ to indicate $r\leq Cs$ for some fixed numeric constant $C>0$.

The (possible) horizontal and vertical edge locations in the ground truth image $u^\dagger$ we denote by
\begin{equation*}
\Lh=\R/\Z\times\{y_1,\ldots,y_N\}\subset\Omega,\qquad
\Lv=\{x_1,\ldots,x_M\}\times\R/\Z\subset\Omega,
\end{equation*}
compare \cref{fig:grid} left.
Furthermore, we will abbreviate
\begin{equation*}
L=\Lh\cup\Lv,\qquad
\Lh_R=B_R(\Lh),\qquad
\Lv_R=B_R(\Lv),\qquad
L_R=B_R(L).
\end{equation*}

\begin{figure}
\centering
\includegraphics{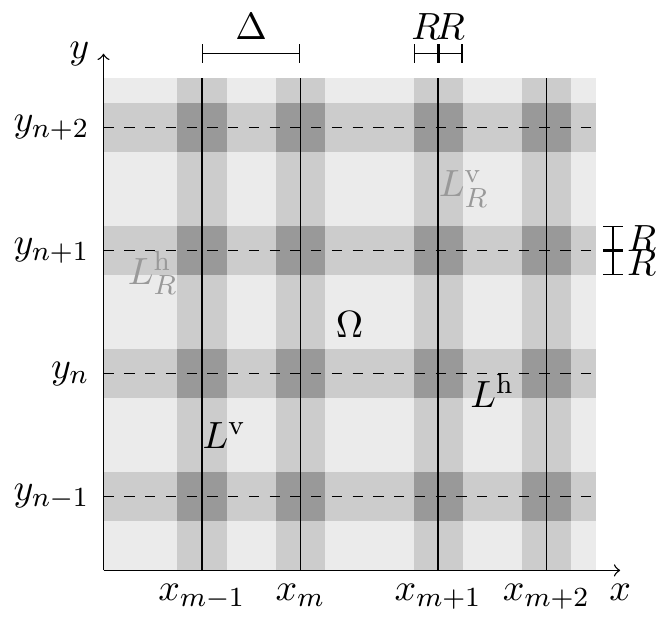}
\hfill
\includegraphics{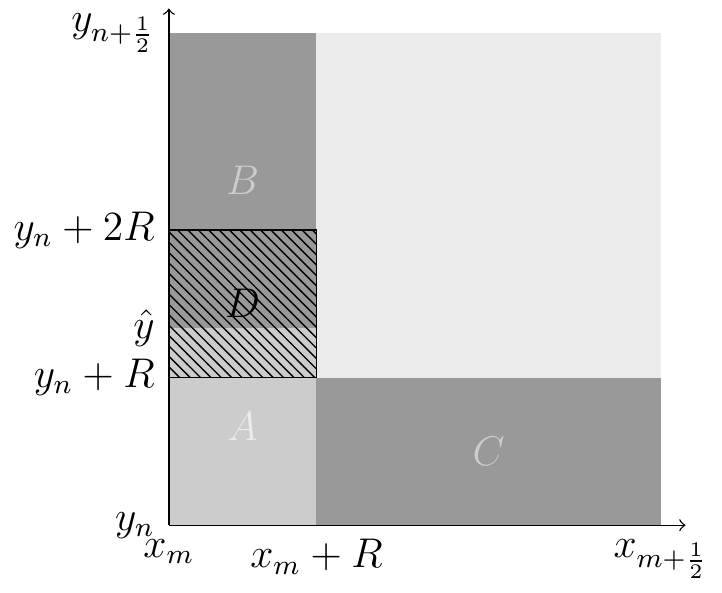}
\caption{Illustration of the edge locations $\Lh$, $\Lv$ together with their neighborhoods $\Lh_R$ and $\Lv_R$ in the full domain $\Omega$ (left plot; introduced in \cref{sec:notation}) and of the decomposition of the bottom left quarter of $[x_m,x_{m+1}]\times[y_n,y_{n+1}]$ into the sets $A,B,C,D$ (right plot; introduced and employed in the proof of \cref{thm:normInterpolation}).}%
\label{fig:grid}
\end{figure}

\notinclude{
\subsection{Domains and sets}

For $k=1,\ldots,m$ and $l=1,\ldots,n$, we further set 
\[ \Lv_{k,l}= \{ (x_k,y) \st y \in ]y_{l},y_{l+1}[ \}, \quad \Lh_{k,l} = \{ (x,y_l) \st x \in ]x_{k},x_{k+1}[ \}
\] 
as well as
\[ \Lvm = \bigcup _{l} \Lv_{k,l}, \quad \Lhn = \bigcup _{k} \Lh_{k,l},
\]
\[ \Lv = \bigcup _{k} \Lv_{k}, \quad \Lh = \bigcup _{l} \Lh_{l},
\]
and
\[ L = \Lv \cup \Lh .
\]
We also denote by $P\Lv$ the one dimensional set obtained by projecting $\Lv$ to the vertical line, by $P\Lv$ the one dimensional set obtained by projecting $\Lv$ to the horizontal line and likewise for $P\Lvm$, $P\Lv_{k,l}$, $P\Lhn$ and $P\Lh_{k,l}$.
In this context, note that $\supp(\D u^\dagger) \subset L$ up to points of $\hd^1$-measure zero.

Further, with $d_{\Sc^1} (s,t) = \min_{k \in \Z} | s - t - k|$ and for $k=1,\ldots,m-1$ and $l=1,\ldots,n-1$, define
\[ \Delta_k = d_{\Sc^1}(x_k,x_{k+1}),\quad \Delta_l= d_{\Sc^1}(y_l,y_{l+1}),
\]
\[ \Delta_m = d_{\Sc^1}(x_m,x_{1}),\quad \Delta_n= d_{\Sc^1}(y_n,y_{1}),
\]
and
\[ \Delta = \min \{ \min_{k=1,\ldots,m} \Delta_k ,\min_{l=1,\ldots,n} \Delta_l \}.
\]
}%

The vector space of infinitely smooth real functions with compact support on a (open or closed) subset $A$ of $\R/\Z$ or $\R$ or $\Omega$ is denoted $C_c^\infty(A)$.
If the functions are $\R^2$-valued we indicate this by $C_c^\infty(A,\R^2)$.
Now let $\omega$ denote a one-dimensional domain ($\R/\Z$ or $\R$ or an interval in $\R$).
For an (open or closed) interval $A\subset\omega$ we define the total variation of a function $u\in L^1(\omega)$ as
\begin{equation*}
\TVp(u;A)= \sup \left \{ \int_A u \varphi'\wrt x \,\middle|\, \varphi \in C_c^\infty(A),\, \|\varphi\|_\infty\leq 1 \right\}.
\end{equation*}
In case of $A=\omega$ we just write $\TVp(u)$.
It is well-known that this total variation is finite if and only if the distributional derivative $u'$ of $u$ is a Radon measure,
in which case the total variation equals the total mass of the measure $u'$, that is, $\TVp(u)=|u'|(\omega)$.
Note that by our above definition, $\TVp(u;A)$ ignores potential discontinuities on the boundary $\partial A$ of the interval $A$, thus $\TVp(u;[a,b])=\TVp(u;]a,b[)=|u'|(]a,b[)$.
We furthermore introduce a weighted total variation: Let $S\subset\R/\Z$, then we set
\begin{equation*}
\TVd{S}(u;A)=\int_A\dist(x,S)^2\wrt|u'|(x).
\end{equation*}
Similarly we define the horizontal, vertical, and total variation in two dimensions:
Denote the (distributional) $x$- and $y$-derivative by $\D_x$ and $\D_y$, respectively, and the gradient by $\D$.
Let $A\subset\Omega$ be open or closed, then we set
\begin{align*}
\TV^\hor(u;A)
&=\sup\left\{\int_A u \D_x\varphi\wrt(x,y) \,\middle|\, \varphi \in C_c^\infty(A),\, \|\varphi\|_\infty\leq 1 \right\}
=|\D_xu|(A),\\
\TV^\ver(u;A)
&=\sup\left\{\int_A u \D_y\varphi\wrt(x,y) \,\middle|\, \varphi \in C_c^\infty(A),\, \|\varphi\|_\infty\leq 1 \right\}
=|\D_yu|(A),\\
\TVp(u;A)
&=\sup\left\{\int_A u \dive\varphi\wrt(x,y) \,\middle|\, \varphi \in C_c^\infty(A,\R^2),\, \sup_{(x,y)\in A}|\varphi(x,y)|_\infty\leq 1 \right\}\\
&=\TV^\hor(u;A)+\TV^\ver(u;A).
\end{align*}
Note that $\TVp$ denotes an anisotropic version of the total variation functional (sometimes known as $\ell^1-\TV$) which prefers vertical and horizontal jumps,
but we do not highlight this in our notation since we will exclusively deal with this version.
As in one space dimension we will also need the weighted versions, so for $S\subset\Omega$ we define
\begin{align*}
\TVd{S}^\hor(u;A)&=\int_A\dist((x,y),S)^2\wrt|\D_xu|(x,y),\\
\TVd{S}^\ver(u;A)&=\int_A\dist((x,y),S)^2\wrt|\D_yu|(x,y),\\
\TVd{S}(u;A)&=\int_A\dist((x,y),S)^2\wrt|\D u|_1(x,y),
\end{align*}
where $|\D u|_1$ is short for $|\D_xu|+|\D_yu|$.
As above, whenever $A=\Omega$ we drop the dependence on $A$ in the notation.
The space of functions of bounded variation (with periodic boundary conditions) is then defined as
\begin{equation*}
\BV(\Omega) =  \{ u \in L^1(\Omega) \,|\, \TVp(u) <\infty \}.
\end{equation*}
As can be shown for instance by adapting the proofs in \cite{holler16tvsubdif_mh} to the case of periodic boundary conditions, the subdifferential of $\TVp$, regared as functional $\TV:L^2(\Omega) \rightarrow [0,\infty]$, is given as
\begin{equation*}
\partial\TVp(u)=\left\{-\dive \varphi\,\middle|\,\varphi\in H(\dive;\Omega),\,\sup_{(x,y)\in\Omega}|\varphi(x,y)|_\infty\leq1,\,T_{\D u}\varphi=1\right\},
\end{equation*}
where the normal trace $T_{\D u}\varphi$ is defined as the unique function in the $|\D u|_1$-weighted $L^1$-space $L^1(\Omega;|\D u|_1)$ satisfying
\begin{equation*}
\int_\Omega T_{\D u}\varphi\psi\,\d|\D u|_1=-\int_\Omega u\dive(\varphi\psi)\,\d x
\end{equation*}
for all $\psi\in C_c^\infty(\Omega)$.

Finally, note that we consider the minimization of $\TVp(u)+\iota_{\{f^\dagger\}}(Ku)$ and $\TVp(u)+\tfrac1{2\alpha}|Ku-f^\delta|_2^2$ over the space $L^2(\Omega)$. By the continuous embedding of $\BV(\Omega)$ in $L^2(\Omega)$ (see \cite[Corollary 3.49]{Ambrosio}), this is equivalent to minimizing over $\BV(\Omega)$.
Consequently, we regard the bounded linear measurement operator $K$ as mapping $K:L^2(\Omega)\to Y$, where $Y$ denotes the finite-dimensional image space endowed with the Euclidean norm. Note also that we denote by $K^\dual:Y^\dual  \rightarrow L^2(\Omega)^\dual$ the dual operator of $K$.

\subsection{Norm interpolation}
In essence, we will later reduce $L^1$-norm estimates to total variation estimates,
which is possible since the total variation is stronger than the $L^1$-norm (in one space dimension even stronger than the $L^\infty$-norm).
The total variation of our reconstruction error, however, will only be small in certain areas, while in other areas we only have control of a \emph{weighted} total variation.
Still we will be able to obtain a bound on the $L^1$-norm by some norm interpolation,
based on the following \namecref{thm:normInterpolationInterval} in one space dimension.

\begin{lem}[Weighted TV interpolation with boundary singularity]\label{thm:normInterpolationInterval}
Let $a>0$ and set $\omega=[0,a]$. If $u\in\BV(\omega)$ satisfies $u(a)=0$ (in the BV-trace sense), then
\begin{equation*}
\|u\|_1 \leq \sqrt{\TVd{\{0\}}(u)\TV(u)}.
\end{equation*}
\end{lem}
\begin{proof}
Denote the derivative of $u$ by $u'$, a Radon measure on $\omega$. We have
$u(x)=-\int_x^a\d u'(t)$
and can therefore estimate
\begin{multline*}
\|u\|_1
=\int_0^a\left|\int_x^a\d u'(t)\right|\,\d x
\leq\int_0^a\int_x^a\d |u'|(t)\,\d x
=\int_0^a\int_0^a\chi_{\{t>x\}}(x)\,\d |u'|(t)\,\d x\\
=\int_0^at\,\d|u'|(t)
\leq\sqrt{\int_0^at^2\,\d|u'|(t)\cdot\int_0^a\d|u'|(t)}
=\sqrt{\TVd{\{0\}}(u)\TV(u)},
\end{multline*}
the last inequality being the Cauchy--Schwarz inequality in the Lebesgue space $L^2$ with respect to the measure $|u'|$.
\end{proof}

\notinclude{
\begin{lem}[Weighted TV interpolation with interior singularity]\label{thm:normInterpolationDoubleInterval}
Let $a<0<b$ and set $\omega=[a,b]$. If $u\in\BV(\omega)$ satisfies $u(a)=u(b)=0$ (in the BV-trace sense), then
\begin{equation*}
\|u\|_1 \leq \sqrt{\TVd{\{0\}}(u)\TV(u)}.
\end{equation*}
\end{lem}
\begin{proof}
Applying the previous \crefname{thm:normInterpolationInterval} separately to $u$ on $[a,0]$ and $[0,b]$ we obtain
\begin{multline*}
\|u\|_{1}
=\|u|_{[a,0]}\|_1+\|u|_{[0,b]}\|_1\\
\leq\sqrt{\TV(u;[a,0])\TVd{\{0\}}(u;[a,0])}+\sqrt{\TV(u;[0,b])\TVd{\{0\}}(u;[0,b])}
\leq\sqrt{\TV(u)\TVd{\{0\}}(u)},
\end{multline*}
where the last step follows from $\sqrt{\alpha\beta}+\sqrt{\gamma\delta}\leq\sqrt{(\alpha+\gamma)(\beta+\delta)}$ for all nonnegative $\alpha,\beta,\gamma,\delta$,
as can be checked by direct computation.
\end{proof}
}%

To extend this one-dimensional result to two space dimensions we require the following statement.

\notinclude{
\begin{lem}[Generic cross-sections]\label{lem:line_from_area_general}
Consider a non-negative Radon measure $\mu_x \prodm{x} \wrt x$ on $[a,b]\times[c,d]$,
that is, a product of one-dimensional measures $\mu_x$ and the Lebesgue measure.
Then there exists a subset $S\subset[a,b]$ of positive Lebesgue measure such that
\begin{equation*}
\int_c^d \wrt \mu_{\hat{x}} (y) \leq  \frac1{b - a}  \int_{a}^{b}\int_c^d \wrt\mu_x (y) \wrt x
\qquad\text{for all }\hat x\in S.
\end{equation*}
\end{lem}
\begin{proof}
Assume the assertion does not hold, then
\[
\int_{a}^{b}\int_c^d \wrt\mu_{\hat x} (y) \wrt\hat x
>\int_{a}^{b}\frac1{b - a}  \int_{a}^{b}\int_c^d \wrt\mu_x (y) \wrt x \wrt\hat x
=\int_{a}^{b}\int_c^d \wrt\mu_{\hat x} (y) \wrt\hat x,
\]
a contradiction.
\end{proof}
}%

\begin{lem}[Generic cross-sections]\label{lem:line_from_area_general}
Let $\varphi_i:[a,b]\to[0,\infty)$ be non-negative measurable functions for $i=1,\ldots,I$.
Then there exists a subset $S\subset[a,b]$ of positive Lebesgue measure such that
\begin{equation*}
\varphi_i(\hat{x}) \leq  \frac I{b - a}  \int_{a}^{b}\varphi_i(x) \wrt x,
\quad i=1,\ldots,I,
\qquad\text{for almost all }\hat x\in S.
\end{equation*}
\end{lem}
\begin{proof}
Without loss of generality, by rescaling the functions $\varphi_i$ we may assume $\frac1{b-a}\int_a^b\varphi_i(x)\wrt x=1$ for all $i=1,\ldots,I$
(the case where the integral vanishes is trivial since it implies $\varphi_i=0$).
Assume the assertion does not hold, that is, $\max\{\varphi_1,\ldots,\varphi_I\}>I$ almost everywhere, then
\[
I
<\frac1{b-a}\int_a^b\max\{\varphi_1(x),\ldots,\varphi_I(x)\}\wrt x
\leq\frac1{b-a}\int_a^b\varphi_1(x)+\ldots+\varphi_I(x)\wrt x
=I,
\]
a contradiction.
\end{proof}

\begin{lem}[Weighted TV interpolation with interior singularity lines]\label{thm:normInterpolation}
Let $R<\Delta/4$.
Any $u\in\BV(\Omega)$ which vanishes outside a compact subset of $L_R$ satisfies
\begin{align*}
\|u\|_1&\leq3\sqrt{\TVp(u)}\left(\sqrt{\TVd{\Lh}^\ver(u;\Lh_{2R})}+\sqrt{\TVd{\Lv}^\hor(u;\Lv_{2R})}\right),\\
\sum_{n=1}^N\|u(\cdot,y_n-R)\|_1+\|u(\cdot,y_n+R)\|_1
&\leq\TV^\ver(u;(\Lh_R)^c)+\tfrac1R\sqrt{\TV^\hor(u)\TVd{\Lv}^\hor(u;\Lv_R)},\\
\sum_{m=1}^M\|u(x_m-R,\cdot)\|_1+\|u(x_m+R,\cdot)\|_1
&\leq\TV^\hor(u;(\Lv_R)^c)+\tfrac1R\sqrt{\TV^\ver(u)\TVd{\Lh}^\ver(u;\Lh_R)},\\
\end{align*}
where $u(\cdot,y_n\pm R)$ and $u(x_m\pm R,\cdot)$ denote the BV-traces of $u$ restricted to $(\Lh_R)^c$ and $(\Lv_R)^c$, respectively.
\end{lem}
\begin{proof}
For $m\in \{1,\ldots,M\}$ and $n \in \{1,\ldots,N\}$
abbreviate $x_{m+\frac12}$ and $y_{n+\frac12}$ to be the midpoints of $[x_m,x_{m+1}]$ and $[y_n,y_{n+1}]$, respectively.
We derive the claimed estimates for the restriction of $u$ onto the bottom left quarter
\[
[x_{m},x_{m+\frac12}] \times [y_n,y_{n+\frac12}]
\]
of $[x_m,x_{m+1}]\times[y_n,y_{n+1}]$, the estimate on the other quarters follows analogously.
From those local estimates, the estimate on the whole domain $\Omega$ follows
using the Cauchy--Schwarz inequality $\sum_{j}\sqrt{a_j}\sqrt{b_j}\leq\sqrt{\sum_{j}a_j}\sqrt{\sum_jb_j}$ for any sequences $a_j,b_j$.

In the below estimates we will frequently consider restrictions of $u$ to horizontal or vertical lines,
which of course has to be understood in the sense of BV-traces.
However, except for $u(\cdot,y_n\pm R)$ and $u(x_m\pm R,\cdot)$ it will not matter from which side of the lines the BV-trace is taken.

First note that by \cref{lem:line_from_area_general} (for $I=1$ and $\varphi_1(y)=\int_{x_m}^{x_m+R} |u(x,y)| \wrt x$) we can choose some $\hat{y} \in ]y_n+R ,y_n + 2R[$ such that%
\begin{align*}
\int_{x_m}^{x_m+R} |u(x,\hat y)| \wrt x
&\leq \frac{1}{R}\int_{y_n+R}^{y_n + 2R}\int_{x_m}^{x_m+R} |u(x,y)| \wrt (x,y)\\
&\leq\frac{1}{R}\int_{y_n+R}^{y_n + 2R}\sqrt{\TV^\hor(u(\cdot,y);[x_m,x_m+R])\TVd{\{x_m\}}^\hor(u(\cdot,y);[x_m,x_m+R])}\wrt y\\
&\leq\frac{1}{R}\sqrt{ \TVp^\hor(u;D)\TVd{\Lv}^\hor(u;D) },
\end{align*}
where the last inequality is H\"older's inequality, and we abbreviated $D=[x_m,x_m+R]\times[y_n+R,y_n+2R]$, see \cref{fig:grid} right.
We further set
\begin{equation*}
A=[x_m,x_m+R]\times[y_n,\hat y],\qquad
B=[x_m,x_m+R]\times[\hat y,y_{n+\frac12}],\qquad
C=[x_m+R,x_{m+\frac12}]\times[y_n,y_n+R]
\end{equation*}
and estimate
\begin{equation*}
\int_{x_m}^{x_{m+\frac12}} \int_{y_n}^{y_{n+\frac12}} |u|\wrt(x,y)
= \int_{A}|u|\wrt(x,y)
+ \int_{B} |u|\wrt(x,y)
+ \int_{C} |u|\wrt(x,y).
\end{equation*}
Applying \cref{thm:normInterpolationInterval}, the last two terms on the right-hand side can be estimated as
\begin{align*}
\int_{B} |u| \wrt(x,y)
& \leq \int_{\hat{y}}^{y_{n+\frac12}} \sqrt{ \TV^\hor(u(\cdot,y);[x_m,x_m+R]) \TVd{\{x_m\}}^\hor(u(\cdot,y);[x_m,x_m+R])}\wrt y \\
& \leq \sqrt{ \TVp^\hor(u;B) \TVd{\Lv}^\hor(u;B)},\\
\int_{C} |u| \wrt(x,y)
& \leq \int_{x_m+R}^{x_{m+\frac12}}  \sqrt{ \TV^\ver(u(x,\cdot);[y_n,y_n+R]) \TVd{\{y_n\}}^\ver(u(x,\cdot);[y_n,y_n+R])} \wrt x  \\
& \leq \sqrt{\TVp^\ver(u;C)\TVd{\Lh}^\ver(u;C)},
\end{align*}
where, again, in each case the second inequality is H\"older's inequality.
The first term can be estimated, again applying \cref{thm:normInterpolationInterval}, as
\begin{align*}
\int_{A} |u|\wrt(x,y)
& \leq \int_{x_m}^{x_m+R}  \int_{y_n}^{\hat{y}}  |u(x,y) - u(x,\hat{y})|\wrt (x,y) + \dist(y_n,\hat{y}) \int_{x_m}^{x_m+R}   |u(x,\hat{y})| \wrt x \\
&  \leq \int_{x_m}^{x_m+R}  \sqrt{ \TV^\ver(u(x,\cdot);[y_n,\hat{y}])\TVd{\{y_n\}}^\ver(u(x,\cdot);[y_n,\hat{y}]) }\wrt x\\
&\hspace*{10ex}+\frac{\dist(y_n,\hat{y})}{R}\sqrt{ \TVp^\hor(u;D)\TVd{\Lh}^\hor(u;D) }\\
& \leq \sqrt{ \TVp^\ver(u;A)\TVd{\Lh}^\ver(u;A) }+2\sqrt{ \TVp^\hor(u;D)\TVd{\Lh}^\hor(u;D) }.
\end{align*}
Summarizing we get
\begin{align*}
\int_{x_m}^{x_{m+\frac12}} \int_{y_n}^{y_{n+\frac12}} |u|\wrt(x,y)
&\leq\sqrt{ \TVp(u;A)\TVd{\Lh}^\ver(u;A) }+\sqrt{ \TVp(u;B) \TVd{\Lv}^\hor(u;B)}\\
&\quad+\sqrt{\TVp(u;C)\TVd{\Lh}^\ver(u;C)}+2\sqrt{ \TVp(u;D)\TVd{\Lh}^\ver(u;D) }\\
&\leq\sqrt{\TVp(u;A\cup C\cup B)}\left(\sqrt{\TVd{\Lh}^\ver(u;A\cup C)}+3\sqrt{\TVd{\Lv}^\hor(u;B\cup D)}\right),
\end{align*}
which proves the first estimate.
For the second estimate (the third follows analogously) we calculate
\begin{multline*}
\int_{x_m}^{x_{m+\frac12}}|u(x,y_n+R)|\wrt x
\leq\int_{x_m}^{x_m+R}|u(x,\hat y)|\wrt x
+\int_{x_m}^{x_m+R}|u(x,y_n+R)-u(x,\hat y)|\wrt x\\
\leq\frac{1}{R}\sqrt{ \TVp^\hor(u;D)\TVd{\Lv}^\hor(u;D) }
+\TV^\ver(u;D).
\qedhere
\end{multline*}
\end{proof}

\subsection{Tikhonov regularization in Banach spaces}
In this section we briefly recall two results for Tikhonov regularization in Banach spaces (see for instance \cite{HollerSchlueterWirth2022,BurgerOsher2004}).
We consider the inverse problem $Ku=f$ with an operator $K:X\to Y$ between Banach spaces, and, consistent with the particular choice of $K$ in the previous section, assume that $X$ is reflexive and that $K$ is linear and bounded (thus weak-to-weak sequentially continuous).

The corresponding Tikhonov regularization is to approximate the solution to the inverse problem by the minimizer of following minimization problem,
consisting of a fidelity $F_f$ for given measurement $f$ and a regularization $G$,
\begin{equation*}
\min_{u\in X}J_\alpha^f(u),
\qquad
J_\alpha^f(u)=\tfrac1\alpha F_f(Ku)+G(u),
\end{equation*}
where $F_f$ and $G$ are proper and convex with $F_f(f)=0$ and $F_f>0$ else, and where $\alpha\geq0$ is a regularization parameter.
For $\alpha=0$ we interpret this as constraint $F_f(Ku)=0$, thus $Ku=f$.

We let $f^\dagger$ be the noise-free measurement and $u^\dagger$ the correct solution to the inverse problem $Ku^\dagger=f^\dagger$.
The results concern the so-called Bregman distance
\begin{equation*}
\BregmanDistance{G}p(u,u^\dagger)=G(u)-(G(u^\dagger)+\langle p,u-u^\dagger\rangle)
\end{equation*}
of the convex functional $G$ with respect to $p\in\partial G(u^\dagger)$.

\begin{thm}[Vanishing Bregman distance for noiseless reconstruction]\label{thm:ZeroBregman}
Let $u^\dagger\in X$ satisfy the source condition $-K^\dual w^\dagger\in\partial G(u^\dagger)$ for some $w^\dagger\in Y^\dual$.
Then any minimizer $u$ of $J_0^{f^\dagger}$ satisfies
\begin{equation*}
\BregmanDistance{G}{-K^\dual w^\dagger}(u,u^\dagger)=0.
\end{equation*}
\end{thm}
\begin{proof}
One readily calculates
\begin{multline*}
\BregmanDistance{G}{-K^\dual w^\dagger}(u,u^\dagger)
=G(u)-G(u^\dagger)-\langle-K^\dual w^\dagger,u-u^\dagger\rangle\\
=G(u)-G(u^\dagger)+\langle w^\dagger,Ku-Ku^\dagger\rangle
=G(u)-G(u^\dagger)
=J_0^{f^\dagger}(u)-J_0^{f^\dagger}(u^\dagger)
\leq0
\end{multline*}
so that the nonnegativity of the Bregman distance implies all above terms to vanish.
\end{proof}

Below, for a functional $F:Y\to\R$ we indicate its convex conjugate by $F^\dual :Y^\dual \to\R$, $F^\dual (w)=\sup_{f\in Y}\langle w,f\rangle-F(f)$.

\begin{rem}[Interpretation of the source condition]
Note that the source condition $-K^\dual w^\dagger\in\partial G(u^\dagger)$ is one of the two necessary and sufficient primal-dual optimality conditions for minimizing $J_0^{f^\dagger}$.
The other one is $Ku^\dagger\in\partial(\tfrac10F_{f^\dagger})^\dual(w^\dagger)=\partial\iota_{\{f^\dagger\}}^\dual(w^\dagger)=\{f^\dagger\}$, thus automatically satisfied.
Thus, if strong duality holds, the source condition implies that $u^\dagger$ minimizes $J_0^{f^\dagger}$, and $w^\dagger$ certifies this.
\end{rem}

Now let $f^\delta$ be a noisy measurement with $F_{f^\delta}(f^\dagger)\leq\delta$, which is how we quantify the noise strength.
In that case the Bregman distance can be estimated by the following.

\begin{thm}[Bregman distance estimate for noisy reconstruction]\label{thm:BregmanEstimate}
Let $u^\dagger\in X$ satisfy the source condition $-K^\dual w^\dagger\in\partial G(u^\dagger)$ for some $w^\dagger\in Y^\dual $.
Then a minimizer $u_\alpha^\delta$ of $J_\alpha^{f^\delta}$ satisfies
\begin{align*}
\BregmanDistance{G}{-K^\dual w^\dagger}(u_\alpha^\delta,u^\dagger)
&\leq\left(3\delta+F_{f^\delta}^\dual (2\alpha w^\dagger)+F_{f^\delta}^\dual (-2\alpha w^\dagger)\right)/(2\alpha),\\
F_{f^\delta}(Ku_\alpha^\delta)
&\leq\left(3\delta+F_{f^\delta}^\dual (2\alpha w^\dagger)+F_{f^\delta}^\dual (-2\alpha w^\dagger)\right),\\
\langle K^\dual w,u_\alpha^\delta-u^\dagger\rangle
&\leq\left(4\delta+F_{f^\delta}^\dual (2\alpha w^\dagger)+F_{f^\delta}^\dual (-2\alpha w^\dagger)+F_{f^\delta}^\dual (2\alpha w)+F_{f^\delta}^\dual (-2\alpha w)\right)/(2\alpha)
\text{ for all }w\in Y^\dual .
\end{align*}
\end{thm}
\begin{proof}
We first identify the Bregman distance with the following term,
\begin{multline*}
\left[G(u_\alpha^\delta)+\langle K^\dual w^\dagger,u_\alpha^\delta\rangle-(\tfrac10F_{f^\dagger})^\dual (w^\dagger)\right]
-\left[G(u^\dagger)+\langle K^\dual w^\dagger,u^\dagger\rangle-(\tfrac10F_{f^\dagger})^\dual (w^\dagger)\right]\\
=G(u_\alpha^\delta)-G(u^\dagger)-\langle-K^\dual w^\dagger,u_\alpha^\delta-u^\dagger\rangle
=\BregmanDistance{G}{-K^\dual w^\dagger}(u_\alpha^\delta,u^\dagger).
\end{multline*}
Exploiting the optimality of $u_\alpha^\delta$, we furthermore have
\begin{equation*}
G(u_\alpha^\delta)+\tfrac1\alpha F_{f^\delta}(Ku_\alpha^\delta)
=J_\alpha^{f^\delta}(u_\alpha^\delta)
\leq J_\alpha^{f^\delta}(u^\dagger)
=G(u^\dagger)+\tfrac1\alpha F_{f^\delta}(f^\dagger)
\leq G(u^\dagger)+\tfrac\delta\alpha.
\end{equation*}
Fenchel's inequality now yields
\begin{equation*}
\langle K^\dual w,u_\alpha^\delta-u^\dagger\rangle
=\frac{\langle2\alpha w,Ku_\alpha^\delta\rangle+\langle-2\alpha w,Ku^\dagger\rangle}{2\alpha}
\leq\frac{F_{f^\delta}(Ku_\alpha^\delta)+F_{f^\delta}^\dual (2\alpha w)+F_{f^\delta}(Ku^\dagger)+F_{f^\delta}^\dual (-2\alpha w)}{2\alpha},
\end{equation*}
which already proves the third statement in case the second holds.
Using both previous inequalities we finally obtain
\begin{multline*}
\BregmanDistance{G}{-K^\dual w^\dagger}(u_\alpha^\delta,u^\dagger)
=\left[G(u_\alpha^\delta)+\langle K^\dual w^\dagger,u_\alpha^\delta\rangle-(\tfrac10F_{f^\dagger})^\dual (w^\dagger)\right]
-\left[G(u^\dagger)+\langle K^\dual w^\dagger,u^\dagger\rangle-(\tfrac10F_{f^\dagger})^\dual (w^\dagger)\right]\\
\leq\frac\delta\alpha-\frac1\alpha F_{f^\delta}(Ku_\alpha^\delta)+\langle w^\dagger,K(u_\alpha^\delta-u^\dagger)\rangle
\leq\frac1{2\alpha}\left(3\delta+F_{f^\delta}^\dual (2\alpha w^\dagger)+F_{f^\delta}^\dual (-2\alpha w^\dagger)\right)-\frac1{2\alpha}F_{f^\delta}(Ku_\alpha^\delta).
\qedhere
\end{multline*}
\end{proof}

\begin{rem}[Rates from estimates]\label{rem:rates}
It is not difficult to deduce that the property $F_{f^\delta}(z)>0$ unless $z=f^\delta$ implies differentiability of $F_{f^\delta}^\dual (\pm2\alpha w)$ in $\alpha=0$. Thus
\begin{equation*}
(F_{f^\delta}^\dual (2\alpha w)+F_{f^\delta}^\dual (-2\alpha w))/\alpha\to0
\end{equation*}
as $\alpha\to0$.
Choosing $\alpha$ as minimizer of $(\delta+F_{f^\delta}^\dual (2\alpha w^\dagger)+F_{f^\delta}^\dual (-2\alpha w^\dagger))/\alpha$ we therefore get a convergence rate.
In this article we will for simplicity consider $F_{f^\delta}(f)=\frac12\|f-f^\delta\|_Y^2$ so that $F_{f^\delta}^\dual (2\alpha w)+F_{f^\delta}^\dual (-2\alpha w)=2\|w\|_{Y^\dual }^2\alpha^2$.
Thus, by choosing $\alpha=\sqrt\delta$, \cref{thm:BregmanEstimate} provides the estimates
\begin{align*}
\BregmanDistance{G}{-K^\dual w^\dagger}(u_\alpha^\delta,u^\dagger)
&\leq C(\|w^\dagger\|_{Y^\dual })\sqrt\delta,\\
F_{f^\delta}(Ku_\alpha^\delta)
&\leq C(\|w^\dagger\|_{Y^\dual })\sqrt\delta,\\
\langle K^\dual w,u_\alpha^\delta-u^\dagger\rangle
&\leq (C(\|w^\dagger\|_{Y^\dual })+C(\|w\|_{Y^\dual }))\sqrt\delta,
\end{align*}
where $C(s)$ denotes a constant only depending on $s\in\R$.
Results for other fidelities follow analogously.
\end{rem}

\section{Exact recovery and convergence rates for reconstruction from noisy data}\label{sec:reconstructionError}

Recall that to solve the inverse problem $Ku=f$ we consider the minimization problem
\begin{equation} \label{eq:main_min_prob}\tag*{$P_\alpha(f)$}
\min_{u\in L^2(\Omega)}J_{\alpha}^{f}(u)
\qquad\text{with }
J_{\alpha}^{f}(u)=\TVp(u)+\tfrac1{2\alpha}|Ku-f|_2^2.
\end{equation}
Here $K:L^2(\Omega)\to Y$ is a bounded, linear operator (thus weak-to-weak sequentially continuous).%

We think of the truncated Fourier transform \eqref{eqn:truncatedFourier}, but the exact form of the operator does actually not matter in this section.
First we will prove that, given perfect, noise-free measurements $f^\dagger=Ku^\dagger$ of the ground truth $u^\dagger$,
then $u^\dagger$ can be exactly reconstructed via $P_0(f^\dagger)$ if it satisfies certain regularity conditions.
Afterwards we prove $L^1$-convergence rates for the approximation of $u^\dagger$ via $P_\alpha(f^\delta)$ in case of noisy data $f^\delta$.

\subsection{Identification of the support of $\D u^\dagger$}

We first show under which conditions the support of $\D u^\dagger$ can be reconstructed.
The existence of the corresponding dual variables for $K$ being the truncated Fourier transform
will be proved in \cref{thm:dualCertificatesI} of \cref{sec:dualCertificates} under just a minimum separation condition.
Based on \cref{ass:consistentGradient} we denote by $s_n^\hor,s_m^\ver\in\{-1,1\}$ the sign of $\D_yu^\dagger\mres\R/\Z\times\{y_n\}$ and of $\D_xu^\dagger\mres\{x_m\}\times\R/\Z$, respectively.

\begin{prop}[Exact recovery of gradient support]\label{thm:exactRecoverySupport}
Let $u^\dagger$ satisfy \cref{ass:consistentGradient} as well as the source condition that there exist $w^\hor,w^\ver\in Y^\dual $ with $K^\dual w^\hor=\D_yg^\hor$, $K^\dual w^\ver=\D_xg^\ver$
for some $g=(g^\ver,g^\hor)\in H(\dive;\Omega)\cap C(\Omega)^2$ satisfying
\begin{alignat*}{4}
g^\hor&=s_n^\hor&&\text{ on }\R/\Z\times\{y_n\}&&\text{ for all }n,&&\quad |g^\hor|<1\text{ on }\Omega\setminus \Lh,\\
g^\ver&=s_m^\ver&&\text{ on }\{x_m\}\times\R/\Z&&\text{ for all }m,&&\quad |g^\ver|<1\text{ on }\Omega\setminus \Lv.
\end{alignat*}
Then the solution $u$ to $P_0(f^\dagger)$ satisfies $|\D u|(\Omega\setminus L)=0$.
\end{prop}
\begin{proof}
Obviously $-K^\dual (w^\hor+w^\ver)=-\dive g\in\partial\TVp(u^\dagger)$.
By \cref{thm:ZeroBregman} we have
\begin{equation*}
0=\BregmanDistance{\TVp}{-\dive g}(u,u^\dagger)
=\TVp(u)-\TVp(u^\dagger)-\langle g,\D u-\D u^\dagger\rangle
=\TVp(u)-\langle g,\D u\rangle.
\end{equation*}
However, if $|Du|(\Omega\setminus L)$ were nonzero, then $|g|_\infty\leq1$ everywhere and $|g|_\infty< 1$ outside $L$ would imply
\begin{equation*}
\TVp(u)-\langle g,\D u\rangle
=|\D_xu|(\Omega\setminus L^\ver)+|\D_xu|(L^\ver)-\langle g^\ver,\D_x u\rangle
+|\D_yu|(\Omega\setminus L^\hor)+|\D_yu|(L^\hor)-\langle g^\hor,\D_y u\rangle
>0,
\end{equation*}
a contradiction.
\end{proof}

\subsection{Identification of values of $u^\dagger$}
Under slightly stronger conditions also the values of $u^\dagger$ can be exactly reconstructed.
The existence of the corresponding dual variables will be proved in \cref{thm:dualCertificatesII} of \cref{sec:dualCertificates} under just a minimum separation condition.

\begin{prop}[Exact recovery]\label{thm:exactRecovery}
Let $u^\dagger$ satisfy \cref{ass:consistentGradient}.
Furthermore, in addition to the source condition from \cref{thm:exactRecoverySupport} assume that for any sign combination $s_{mn}\in\{-1,1\}$ there exists $w\in Y^\dual $ with
\begin{equation*}
s_{mn}\int_{x_m}^{x_{m+1}}\int_{y_n}^{y_{n+1}}K^\dual w(x,y)\wrt(x,y)>0
\end{equation*}
for all $m=1,\ldots,M$, $n=1,\ldots,N$, then $u^\dagger$ is the unique solution to $P_0(f^\dagger)$.
\end{prop}
\begin{proof}
By \cref{thm:exactRecoverySupport} there exist $(u_{mn})_{m,n} \in \R^{M\times N}$ such that
\begin{equation*}
u = \sum_{m=1}^M \sum_{n=1}^N u_{mn} \chi_{[x_m,x_{m+1}]\times[y_n,y_{n+1}]}.
\end{equation*}
Now let $s_{mn}$ be the sign of $u_{mn}-u_{mn}^\dagger$ and $w\in Y^\dual $ the associated dual variable.
Then we can calculate
\begin{multline*}
0
=\langle w,K(u-u^\dagger)\rangle
=\langle K^\dual w,u-u^\dagger\rangle
=\sum_{m=1}^M\sum_{n=1}^N(u_{mn}-u_{mn}^\dagger)\int_{x_m}^{x_{m+1}}\int_{y_n}^{y_{n+1}}K^\dual w(x,y)\wrt(x,y)\\
=\sum_{m=1}^M\sum_{n=1}^N(u_{mn}-u_{mn}^\dagger)s_{mn}\left|\int_{x_m}^{x_{m+1}}\int_{y_n}^{y_{n+1}}K^\dual w(x,y)\wrt(x,y)\right|
>0
\end{multline*}
unless $u=u^\dagger$.
\end{proof}

\begin{rem}[Alternative approaches]
For $K$ being the truncated Fourier series, as we consider in this article,
one can obtain the exactness of reconstructed function values by alternative means. For instance, this can be done using results for the de Prony method \cite{Plonka2013prony_spline} or, already given the correct support of $Du$, by arguing with injectivity of the band-limited Fourier transform on this type of functions.
\end{rem}

\subsection{Convergence rates in the $L^1$-distance}

The correct identification of the support of $\D u^\dagger$ in the noise-free setting required the existence of dual variables
having unit absolute value on the ground truth discontinuity set $L$ and strictly smaller absolute value away from $L$.
In case of noisy data, in which $L$ will not be exactly recovered, we estimate the localization error of the discontinuity lines
using the $\dist(\cdot,L)$-weighted total variation.
This requires the dual variables to additionally diverge sufficiently fast from the value $\pm1$ on $L$,
which, as before, will be true under a minimium separation condition as shown in \cref{thm:dualCertificatesI} of \cref{sec:dualCertificates}.

\notinclude{
\begin{prop}[Bound of total variation by Bregman distance] \label{prop:tv_tvd_estimates}
 Assume that $g = (g^\ver,g^\hor)$ is such that $- \dive g \in \subdif \TVp(u^\dagger)$ and 
\begin{align*}
|g^\ver(x,y)| \leq 1 - \kappa \min(R,\dist(x,\Lv)^2), \\
|g^\hor(x,y)| \leq  1 - \kappa \min(R,\dist(y,\Lh)^2)
\end{align*}
for $0<R < \Delta/2$  and all $(x,y) \in \Omega$. Then
\begin{align*}
\TV^\hor|_{(\Lv_R)^c}(u^\delta)+\TV^\ver|_{(\Lh_R)^c}(u^\delta)
&\leq\BregmanDistance{\TVp}{-\dive g}(u^\delta,u^\dagger)/(\kappa R^2), \\
\TVd{\Lv}^\hor|_{\Lv_R}(u^\delta)+\TVd{\Lh}^v|_{\Lh_R}(u^\delta)
&\leq\BregmanDistance{\TVp}{-\dive g}(u^\delta,u^\dagger)/\kappa.
\end{align*}
\begin{proof}
\begin{multline*}
\kappa R^2 | \D_x u^\delta |((\Lv_R)^c) + \int_{\Lv_R} \kappa\dist(x,\Lv)^2|\D_x u^\delta| + \kappa R^2 | \D_y u^\delta |((\Lh_R)^c) + \int_{\Lh_R} \kappa\dist(y,\Lh)^2|\D_y u^\delta|  \\
 \leq \int_{\Omega} 1 - |g^\ver| \wrt |\D_x u^\delta| + \int_{\Omega} 1 - |g^\hor| \wrt |\D_y u^\delta|  \leq \int_{\Omega} |\D u^\delta|_1 - \langle g,\D u^\delta \rangle   \\
= \int_{\Omega} |\D u^\delta| _1 - |\D u^\dagger |_1 - \langle g,\D u^\delta - \D u^\dagger \rangle = \BregmanDistance{\TVp}	{-\dive g}(u,u^\dagger)
\end{multline*}
\end{proof}
\end{prop}
}%

\begin{prop}[Total variation estimates]\label{thm:TVestimates}
Let $u^\dagger$ satisfy \cref{ass:consistentGradient}.
Furthermore, in addition to the source condition from \cref{thm:exactRecoverySupport} assume that $g^\hor,g^\ver$ satisfy
\begin{align*}
|g^\ver(x,y)| &\leq 1 - \kappa \min(R,\dist(x,\Lv))^2, \\
|g^\hor(x,y)| &\leq  1 - \kappa \min(R,\dist(y,\Lh))^2
\end{align*}
for some $\kappa>0$, $0<R < \Delta/4$  and all $(x,y) \in \Omega$. Then any solution $u^\delta$ to $P_{\sqrt\delta}(f^\delta)$ satisfies
\begin{align*}
\TV^\hor(u^\delta;(\Lv_R)^c)+\TV^\ver(u^\delta;(\Lh_R)^c)
&\leq C(\|w^\hor+w^\ver\|_{Y^\dual })\sqrt\delta/(\kappa R^2), \\
\TVd{\Lv}^\hor(u^\delta;\Lv_R)+\TVd{\Lh}^\ver(u^\delta;\Lh_R)
&\leq C(\|w^\hor+w^\ver\|_{Y^\dual })\sqrt\delta/\kappa
\end{align*}
for a constant $C(\|w^\hor+w^\ver\|_{Y^\dual })$ only depending on $\|w^\hor+w^\ver\|_{Y^\dual }$.
\end{prop}
\begin{proof}
Recall from the proof of \cref{thm:exactRecoverySupport} that $-K^\dual (w^\hor+w^\ver)=-\dive g\in\partial\TVp(u^\dagger)$.
Now we can calculate
\begin{align*}
&\kappa R^2\TV^\hor(u^\delta;(\Lv_R)^c)
+\kappa\TVd{\Lv}^\hor(u^\delta;\Lv_R)
+\kappa R^2\TV^\ver(u^\delta;(\Lh_R)^c)
+\kappa\TVd{\Lh}^\ver(u^\delta;\Lh_R)\\
&=\kappa R^2 | \D_x u^\delta |((\Lv_R)^c) + \int_{\Lv_R} \kappa\dist(x,\Lv)^2|\D_x u^\delta| + \kappa R^2 | \D_y u^\delta |((\Lh_R)^c) + \int_{\Lh_R} \kappa\dist(y,\Lh)^2|\D_y u^\delta|  \\
&\leq \int_{\Omega} 1 - |g^\ver| \wrt |\D_x u^\delta| + \int_{\Omega} 1 - |g^\hor| \wrt |\D_y u^\delta| \\
&\leq |\D u^\delta|_1(\Omega) - \langle g,\D u^\delta \rangle  %
= \TVp(u^\delta) - \TVp(u^\dagger) - \langle g,\D u^\delta - \D u^\dagger \rangle%
= \BregmanDistance{\TVp}{-\dive g}(u^\delta,u^\dagger).
\end{align*}
The result now follows from \cref{rem:rates}.
\end{proof}

Note that, in case the assumptions in the above proposition on $g$ hold for some $R>0$, they also hold for any $0<\tilde{R}<R$,
in particular for $\tilde R=\frac R2$.
To derive an $L^1$-convergence rate we will decompose the reconstruction error as
\begin{equation*}
\|u^\delta-u^\dagger\|_1\leq\|u^\delta-\tilde u^\delta\|_1+\|\tilde u^\delta-u^\dagger\|,
\end{equation*}
where $\tilde u^\delta$ is an approximation of $u^\delta$ whose (potentially spurious) discontinuities near $L$ are all relocated to lie exactly on the ground truth discontinuity lines $L$.
To this end we introduce points
\begin{equation*}
x_m^-\in[x_m-\tilde R,x_m-\tfrac{\tilde R}2],\quad
x_m^+\in[x_m+\tfrac{\tilde R}2,x_m+\tilde R],\quad
y_n^-\in[y_n-\tilde R,y_n-\tfrac{\tilde R}2],\quad
y_n^+\in[y_n+\tfrac{\tilde R}2,y_n+\tilde R]
\end{equation*}
for all $m=1,\ldots,M$ and $n=1,\ldots,N$.
These define horizontal and vertical strips
\begin{equation*}
\Lh_\pm=\bigcup_{n=1}^N\R/\Z\times[y_n^-,y_n^+],\quad
\Lv_\pm=\bigcup_{m=1}^M[x_m^-,x_m^+]\times\R/\Z,
\end{equation*}
and we then set $\tilde u^\delta=u^\delta$ on $(\Lh_\pm\cup\Lv_\pm)^c$ and extend it constantly in direction normal to the boundary of $(\Lh_\pm\cup\Lv_\pm)^c$,
\begin{equation*}
\tilde u^\delta(x,y)
=u^\delta(\pi_{[x_m^+,x_{m+1}^-]\times[y_n^+,y_{n+1}^-]}(x,y))
\qquad\text{if }(x,y)\in[x_m,x_{m+1}]\times[y_n,y_{n+1}],
\end{equation*}
where $\pi_A:\Omega\to A$ denotes the orthogonal projection onto the closed convex set $A\subset\Omega$.
We choose the points $x_m^\pm$ and $y_n^\pm$ such that,
abbreviating $P\Lv_{\tilde R/2}=B_{\tilde R/2}(\{x_1,\ldots,x_M\})$ and $P\Lh_{\tilde R/2}=B_{\tilde R/2}(\{y_1,\ldots,y_N\})$, it holds
\begin{align*}
\int_{V_i}\wrt|\D_yu^\delta|(x_m^+,y)
&\leq\frac4{\tilde R}\int_{x_m+\frac{\tilde R}2}^{x_m+\tilde R}\int_{V_i}\wrt|\D_yu^\delta|(x,y)
&&\text{for }V_1=\R/\Z\text{ and }V_2=(\R/\Z)\setminus P\Lh_{\tilde R/2},\\
\int_{V_i}\wrt|\D_yu^\delta|(x_m^-,y)
&\leq\frac4{\tilde R}\int_{x_m-\tilde R}^{x_m-\frac{\tilde R}2}\int_{V_i}\wrt|\D_yu^\delta|(x,y)
&&\text{for }V_1=\R/\Z\text{ and }V_2=(\R/\Z)\setminus P\Lh_{\tilde R/2},\\
\int_{V_i}\wrt|\D_xu^\delta|(x,y_n^+)
&\leq\frac4{\tilde R}\int_{y_n+\frac{\tilde R}2}^{y_n+\tilde R}\int_{V_i}\wrt|\D_xu^\delta|(x,y)
&&\text{for }V_1=\R/\Z\text{ and }V_2=(\R/\Z)\setminus P\Lv_{\tilde R/2},\\
\int_{V_i}\wrt|\D_xu^\delta|(x,y_n^-)
&\leq\frac4{\tilde R}\int_{y_n-\tilde R}^{y_n-\frac{\tilde R}2}\int_{V_i}\wrt|\D_xu^\delta|(x,y)
&&\text{for }V_1=\R/\Z\text{ and }V_2=(\R/\Z)\setminus P\Lv_{\tilde R/2}
\end{align*}
for $m=1,\ldots,M$ and $n=1,\ldots,N$, which is always possible by \cref{lem:line_from_area_general}
(for $I=2$ and $\varphi_i(x)=\int_{V_i}\wrt|\D_yu^\delta|(x,y)$ or $\varphi_i(y)=\int_{V_i}\wrt|\D_xu^\delta|(x,y)$).
The resulting modification $\tilde u^\delta$ satisfies the same bounds as the original $u^\delta$.

\begin{lem}[Total variation estimates for modification]\label{thm:TVMod}
Under the conditions of \cref{thm:TVestimates} we have
\begin{equation*}
\TV^\hor(\tilde u^\delta)\leq5\TV^\hor(u^\delta),\qquad
\TV^\ver(\tilde u^\delta)\leq5\TV^\ver(u^\delta),\qquad
\TV(\tilde u^\delta)\leq5\TV(u^\delta)
\end{equation*}
as well as
\begin{align*}
\TV^\hor(\tilde u^\delta;(\Lv_R)^c)+\TV^\ver(\tilde u^\delta;(\Lh_R)^c)
&\leq C(\|w^\hor+w^\ver\|_{Y^\dual })\sqrt\delta/(\kappa R^2), \\
\TVd{\Lv}^\hor(\tilde u^\delta;\Lv_R)+\TVd{\Lh}^\ver(\tilde u^\delta;\Lh_R)
&\leq C(\|w^\hor+w^\ver\|_{Y^\dual })\sqrt\delta/\kappa
\end{align*}
for a constant $C(\|w^\hor+w^\ver\|_{Y^\dual })$ only depending on $\|w^\hor+w^\ver\|_{Y^\dual }$.
\end{lem}
\begin{proof}
By the choice of the $x_m^\pm$ and $y_n^\pm$ we have
\begin{multline*}
\TV^\hor(\tilde u^\delta)
=\TV^\hor(\tilde u^\delta;(\Lh_\pm)^c)
+\TV^\hor(\tilde u^\delta;\Lh_\pm)
\leq\TV^\hor(u^\delta;(\Lh_\pm)^c)\\
+\sum_{n=1}^N\sum_{s\in\{+,-\}}\tilde R\int_{\R/\Z}\wrt|\D_xu^\delta|(x,y_n^s)
\leq\TV^\hor(u^\delta;(\Lh_\pm)^c)
+4\TV^\hor(u^\delta;\Lh_{\tilde R})
\leq5\TV^\hor(u^\delta).
\end{multline*}
The estimate for $\TV^\ver$ and $\TVp$ follows analogously.
By a similar argument,
\begin{align*}
\TVd{\Lv}^\hor(\tilde u^\delta;\Lv_R)
&=\TVd{\Lv}^\hor(\tilde u^\delta;\Lv_{2\tilde R}\setminus\Lh_\pm)
+\TVd{\Lv}^\hor(\tilde u^\delta;\Lv_{2\tilde R}\cap\Lh_\pm)\\
&\leq\TVd{\Lv}^\hor(u^\delta;\Lv_{2\tilde R}\setminus\Lh_\pm)
+\sum_{n=1}^N\sum_{s\in\{+,-\}}\tilde R\int_{(\R/\Z)\setminus P\Lv_{\tilde R/2}}(2\tilde R)^2\wrt|\D_xu^\delta|(x,y_n^s)\\
&\leq\TVd{\Lv}^\hor(u^\delta;\Lv_{2\tilde R}\setminus\Lh_\pm)
+2(2\tilde R)^2\TV^\hor(u^\delta;\Lh_{\tilde R}\setminus\Lv_{\tilde R/2})\\
&\leq\TVd{\Lv}^\hor(u^\delta;\Lv_{2\tilde R})
+2(2\tilde R)^2\TV^\hor(u^\delta;(\Lv_{\tilde R/2})^c)
\end{align*}
so that \cref{thm:TVestimates} thus implies (recall $2\tilde R=R$)
\begin{equation*}
\TVd{\Lv}^\hor(\tilde u^\delta;\Lv_R)
\leq C(\|w^\hor+w^\ver\|_{Y^\dual })\sqrt\delta/\kappa.
\end{equation*}
The argument for $\TVd{\Lh}^\ver$ is analogous.
Finally,
\begin{align*}
\TV^\hor(\tilde u^\delta;(\Lv_R)^c)
&=\TV^\hor(\tilde u^\delta;(\Lv_{2\tilde R})^c\setminus\Lh_\pm)
+\TV^\hor(\tilde u^\delta;(\Lv_{2\tilde R})^c\cap\Lh_\pm)\\
&\leq\TV^\hor(u^\delta;(\Lv_{2\tilde R})^c\setminus\Lh_\pm)
+\sum_{n=1}^N\sum_{s\in\{+,-\}}\tilde R\int_{(\R/\Z)\setminus P\Lv_{\tilde R/2}}\wrt|\D_xu^\delta|(x,y_n^s)\\
&\leq\TV^\hor(u^\delta;(\Lv_{2\tilde R})^c\setminus\Lh_\pm)
+4\TV^\hor(u^\delta;\Lh_{\tilde R}\setminus\Lv_{\tilde R/2})\\
&\leq\TV^\hor(u^\delta;(\Lv_{2\tilde R})^c)
+4\TV^\hor(u^\delta;(\Lv_{\tilde R/2})^c)
\end{align*}
so that again \cref{thm:TVestimates} (this time using two different values for $R$) implies
\begin{equation*}
\TV^\hor(\tilde u^\delta;(\Lv_R)^c)
\leq C(\|w^\hor+w^\ver\|_{Y^\dual })\sqrt\delta/(\kappa R^2)
\end{equation*}
after potentially adapting the constant. The estimate for $\TV^\ver$ follows in the same way.
\end{proof}

The difference between $u^\delta$ and $\tilde u^\delta$ can now be estimated applying \cref{thm:normInterpolation} to $u^\delta-\tilde u^\delta$.

\begin{prop}[Bound on modification]\label{thm:differenceToModification}
Under the conditions of \cref{thm:TVestimates} we have
\begin{equation*}
\|u^\delta-\tilde u^\delta\|_1
\leq\sqrt{\TVp(u^\dagger)+\sqrt\delta}\,C(\|w^\hor+w^\ver\|_{Y^\dual })\delta^{1/4}/\sqrt{\kappa}
\end{equation*}
for a constant $C(\|w^\hor+w^\ver\|_{Y^\dual })$ only depending on $\|w^\hor+w^\ver\|_{Y^\dual }$.
\end{prop}
\begin{proof}
\Cref{thm:normInterpolation} implies
\begin{multline*}
\|u^\delta-\tilde u^\delta\|_1
\leq3\sqrt{\TVp(u^\delta-\tilde u^\delta)}\left(\sqrt{\TVd{\Lh}^\ver(u^\delta-\tilde u^\delta;\Lh_{2\tilde R})}+\sqrt{\TVd{\Lv}^\hor(u^\delta-\tilde u^\delta;\Lv_{2\tilde R})}\right)\\
\leq3\sqrt{\TVp(u^\delta)+\TVp(\tilde u^\delta)}\left(\sqrt{\TVd{\Lh}^\ver(u^\delta;\Lh_R)+\TVd{\Lh}^\ver(\tilde u^\delta;\Lh_R)}+\sqrt{\TVd{\Lv}^\hor(u^\delta;\Lv_R)+\TVd{\Lv}^\hor(\tilde u^\delta;\Lv_R)}\right)
\end{multline*}
so that from \cref{thm:TVMod} we obtain
\begin{equation*}
\|u^\delta-\tilde u^\delta\|_1
\leq\sqrt{\TVp(u^\delta)}C(\|w^\hor+w^\ver\|_{Y^\dual })\delta^{1/4}/\sqrt\kappa
\end{equation*}
with a constant only depending on $\|w^\hor+w^\ver\|_{Y^\dual }$.
Since $\TV(u^\delta)\leq\TV(u^\dagger)+\frac1{\sqrt\delta}F_{f^\delta}(f)\leq\TV(u^\dagger)+\sqrt\delta$ (with $F_{f^\delta}(f)=\frac12|f-f^\delta|_2^2$) we arrive at the claimed estimate.
\end{proof}

So far we estimated the error contribution from the incorrect localization of the discontinuities in the reconstruction.
As in the case of exact reconstruction, to examine the values of the reconstruction away from $L$ we require an additional dual variable
whose existence will be proved in \cref{thm:dualCertificatesIII} of \cref{sec:dualCertificates} under just a minimum separation condition.
To simplify the notation we abbreviate %
the orthogonal projection onto a closed set $A$ by $\pi_A$.

\begin{prop}[Error estimate for modification]\label{thm:modificationError}
In addition to the source conditions from \cref{thm:TVestimates}
assume the following source conditions:
\begin{enumerate}
\item
There exists $v\in Y^\dual $ with $1=K^\dual v$.
\item
For any combination of signs $s^\hor_{mn},s^\ver_{mn}\in\{-1,1\}$, $m=1,\ldots,M$, $n=1,\ldots,N$,
there exist $v^\hor,v^\ver\in Y^\dual $ with $-\D_yg^\hor = K^\dual v^\hor$, $-\D_xg^\ver = K^\dual v^\ver$ for some $g^\ver,g^\hor\in C(\Omega)$
satisfying $\|g^\hor\|_\infty,\|g^\ver\|_\infty \leq 1$,
\begin{align*}
|g^\hor(x,y) - g^\hor(x,\pi_{\Lh}(y))| &\leq \eta\dist(y,\Lh)^2 \text{ on }\Lh_R,\\
|g^\ver(x,y) - g^\ver(\pi_{\Lv}(x),y)| &\leq \eta\dist(x,\Lv)^2 \text{ on }\Lv_R
\end{align*}
and (for some fixed $C>0$, independent of $s^\hor_{mn},s^\ver_{mn}$)
\begin{align*}
s^\hor_{mn}\int_{x_m}^{x_{m+1}} g^\hor(x,y_n) \wrt x &\geq C,&m=1,\ldots,M,\,n=1,\ldots,N,\\
s^\ver_{mn}\int_{y_n}^{y_{n+1}} g^\ver(x_m,y) \wrt y &\geq C,&m=1,\ldots,M,\,n=1,\ldots,N,
\end{align*}
where the norm of $v^\hor,v^\ver$ shall be bounded, $\|v^\hor\|_{Y^\dual },\|v^\ver\|_{Y^\dual }\leq\frac1C$.
\end{enumerate}
Then
\begin{multline*}
 \|\tilde{u}^\delta-u^\dagger \|_1
 \leq\frac{C(\frac1C)C(\|w^\hor+w^\ver\|_{Y^\dual })}{\sqrt\kappa R}\sqrt{\TVp(u^\dagger)+\sqrt\delta}\,\delta^{1/4}\\+
 (\tfrac1{\kappa R^3}+\tfrac\eta\kappa+1)C(\tfrac1C)\left[C(\|w^\hor+w^\ver\|_{Y^\dual })+C(\|v\|_{Y^\dual })\right]\sqrt\delta+\|\tilde u^\delta-u^\delta\|_1,
\end{multline*}
where $C(s)$ denotes a constant depending only and monotonically on $s$.
\end{prop}
\begin{proof}
Abbreviating by $\av u$ the average value of an integrable function $u$ on $\Omega$, we estimate
\begin{align*}
\|\tilde u^\delta-u^\dagger\|_1
&\leq\|\tilde u^\delta-u^\dagger-\av(\tilde u^\delta-u^\dagger)\|_1
+|\av(u^\delta-u^\dagger)|
+|\av(u^\delta-\tilde u^\delta)|\\
&\leq\TVp(\tilde u^\delta-u^\dagger)
+|\av(u^\delta-u^\dagger)|
+\|u^\delta-\tilde u^\delta\|_1,
\end{align*}
where we used Poincar\'e's inequality.
We can further estimate
\begin{equation*}
|\av(u^\delta-u^\dagger)|
=|\langle K^\dual v,u^\delta-u^\dagger\rangle|
\leq(C(\|w^\hor+w^\ver\|_{Y^\dual })+C(\|v\|_{Y^\dual }))\sqrt\delta
\end{equation*}
by \cref{rem:rates}.
Thus it only remains to estimate $\TVp(\tilde u^\delta-u^\dagger)=\TV^\hor(\tilde u^\delta-u^\dagger)+\TV^\ver(\tilde u^\delta-u^\dagger)$,
of which we just consider the first summand (the second is treated analogously).
To this end let us abbreviate by $[\tilde{u}^\delta-u^\dagger]_x$ the Radon--Nikodym derivative of $\D_x(\tilde{u}^\delta-u^\dagger)$
with respect to the one-dimensional Hausdorff measure $\hd^1\mres\Lv$ on $\Lv$;
it equals the function value jump across $\Lv$.
Now let $h^\ver:\Lv\to\R$ be the piecewise constant function
which on $\{x_m\}\times[y_n,y_{n+1}]$ takes as value the average of $[\tilde{u}^\delta-u^\dagger]_x$ on that segment,
and let $s^\ver_{mn}$ be its sign on that segment. We have
\begin{equation*}
\TV^\hor(\tilde{u}^\delta-u^\dagger)
\leq\int_{\Lv} |h^\ver| \wrt\hd^1+\int_{\Lv} \left|[\tilde{u}^\delta-u^\dagger]_x-h^\ver\right| \wrt\hd^1+\TV^\hor(u^\delta;(\Lv_\pm)^c).
\end{equation*}
The first summand can be estimated (again letting $y_{n+\frac12}$ denote the midpoint of $[y_n,y_{n+1}]$) via
\begin{multline*}
\int_{\Lv} |h^\ver| \wrt\hd^1
=\sum_{m=1}^M\sum_{n=1}^Ns^\ver_{mn}h^\ver(x_m,y_{n+\frac12})
\leq\frac1C\sum_{m=1}^M\sum_{n=1}^Nh^\ver(x_m,y_{n+\frac12})\int_{y_n}^{y_{n+1}}g^\ver(x_m,y)\wrt y\\
\leq\frac1C\int_{\Lv} g^\ver h^\ver\wrt\hd^1
=\frac1C\int_{\Lv_\pm} g^\ver\wrt\D_x (\tilde{u}^\delta-u^\dagger)+\frac1C\int_{\Lv} g^\ver(h^\ver-[\tilde{u}^\delta-u^\dagger]_x)\wrt\hd^1\\
\leq \frac1C\int_{\Lv_\pm} g^\ver\wrt\D_x (\tilde{u}^\delta-u^\dagger)+\frac1C\int_{\Lv} \left|[\tilde{u}^\delta-u^\dagger]_x-h^\ver\right| \wrt\hd^1
\end{multline*}
so that, using Poincar\'e's inequality
\begin{align*}
\TV^\hor(\tilde{u}^\delta-u^\dagger)
&\leq \frac1C\int_{\Lv_\pm} g^\ver\wrt\D_x (\tilde{u}^\delta-u^\dagger)+\left(\frac1C+1\right)\int_{\Lv} \left|[\tilde{u}^\delta-u^\dagger]_x-h^\ver\right| \wrt\hd^1+\TV^\hor(u^\delta;(\Lv_\pm)^c)\\
&\leq \frac1C\int_{\Lv_\pm} g^\ver\wrt\D_x (\tilde{u}^\delta-u^\dagger)+\left(\frac1C+1\right)\sum_{m=1}^M\sum_{n=1}^N\int_{y_n}^{y_{n+1}}\left|\D_y[\tilde{u}^\delta-u^\dagger]_x(x_m,y)\right|\wrt y+\TV^\hor(u^\delta;(\Lv_\pm)^c)\\
&=\frac1C\int_\Omega g^\ver \wrt \D _x (\tilde{u}^\delta-u^\dagger) - \frac1C\int_{(\Lv_{\pm})^c} g^\ver \wrt \D_x (\tilde{u}^\delta-u^\dagger)\\
&\quad+\left(\frac1C+1\right)\sum_{m=1}^M\int_{(\R/\Z)\setminus P\Lh_{\tilde R/2}} \wrt|\D_y u^\delta(x_m^+,y)-\D_y u^\delta(x_m^-,y)|+\TV^\hor(u^\delta;(\Lv_\pm)^c)\\
&\leq \frac1C\int_\Omega g^\ver\wrt \D _x (\tilde{u}^\delta-u^\dagger)+ \frac1C\TV^\hor(u^\delta;(\Lv_\pm)^c) \\
&\quad+\left(\frac1C+1\right)\sum_{m=1}^M\int_{\R/\Z \setminus P\Lh_{\tilde R/2}} \wrt (|\D_y u^\delta|(x_m^+,y)+|\D_y u^\delta|(x_m^-,y))+\TV^\hor(u^\delta;(\Lv_\pm)^c)\\
&\leq \frac1C\int_\Omega g^\ver\wrt \D _x (\tilde{u}^\delta-u^\dagger)  + \left(\frac1C+1\right)\TV^\hor(u^\delta;(\Lv_{\tilde R/2})^c)+\left(\frac1C+1\right)\frac{4}{\tilde R}\TV^\ver(u^\delta;(\Lh_{\tilde R/2})^c),
\end{align*}
where in the last step we exploited the choice of the $x_m^\pm$.
We further estimate
\begin{align*}
\int_\Omega g^\ver \wrt \D _x (\tilde{u}^\delta-u^\dagger) 
& = \int_\Omega g^\ver \wrt \D _x (u^\delta-u^\dagger) + \int_{L_R} g^\ver \wrt \D _x (\tilde{u}^\delta-u^\delta) \\
& \leq \langle K^\dual v^\ver,u^\delta - u^\dagger\rangle + \int_{\Lv_R} g^\ver \wrt \D _x (\tilde{u}^\delta-u^\delta) + \int_{L_R  \setminus \Lv_R} g^\ver \wrt \D _x (\tilde{u}^\delta-u^\delta) \\ 
&  \leq \langle K^\dual v^\ver,u^\delta - u^\dagger\rangle +\int_{\Lv_R} g^\ver \wrt \D _x (\tilde{u}^\delta-u^\delta)  +  \TV^\hor(u^\delta;(\Lv_R)^c).
\end{align*}
Finally, abbreviating $e^\delta=\tilde u^\delta-u^\delta$ we have
\begin{align*}
\int_{\Lv_R} g^\ver \wrt \D_x e^\delta 
& = \sum_{m=1}^M \int_{x_m-R}^{x_m+R}  \int_0^1g^\ver(x_m,y) + [g^\ver(x,y) - g^\ver(x_m,y)] \wrt \D_x e^\delta(x,y)\\
& \leq  \sum_{m=1}^M \int_0^1g^\ver(x_m,y) \left(e^\delta(x_m+R,y)-e^\delta(x_m-R,y)\right) \wrt y\\
&\quad+ \int_{x_m-R}^{x_m+R} \int_0^1 \eta\dist(x,x_m)^2  \wrt |\D_x e^\delta|(x,y) \\
& \leq  \sum_{m=1}^M\sum_{r\in\{-R,R\}}  \|e^\delta(x_m+r,\cdot)\|_1 + \eta\TVd{\Lv}^\hor(e^\delta;\Lv_{R}) \\
& \leq  \TV^\hor(e^\delta;(\Lv_R)^c)+\tfrac1R\sqrt{\TV^\ver(e^\delta)\TVd{\Lh}^\ver(e^\delta;\Lh_R)} + \eta\TVd{\Lv}^\hor(e^\delta;\Lv_R)
\end{align*}
where at the end we used the third estimate from \cref{thm:normInterpolation}.
In summary, we have arrived at
\begin{multline*}
\TV^\hor(\tilde u^\delta-u^\dagger)
\leq\tfrac1C\langle K^\dual v^\ver,u^\delta - u^\dagger\rangle
+(\tfrac3C+1)\TV^\hor(u^\delta;(\Lv_{\tilde R/2})^c)
+\tfrac1C\TV^\hor(\tilde u^\delta;(\Lv_{\tilde R/2})^c)\\
+\frac{4(\frac1C+1)}{\tilde R}\TV^\ver(u^\delta;(\Lh_{\tilde R/2})^c)
+\frac1{CR}\sqrt{\TV^\ver(u^\delta)+\TV^\ver(\tilde u^\delta)}\sqrt{\TVd{\Lh}^\ver(u^\delta;\Lh_R)+\TVd{\Lh}^\ver(\tilde u^\delta;\Lh_R)}\\
+\tfrac\eta C\TVd{\Lv}^\hor(u^\delta;\Lv_R)
+\tfrac\eta C\TVd{\Lv}^\hor(\tilde u^\delta;\Lv_R),
\end{multline*}
which together with \cref{thm:TVestimates,thm:TVMod,rem:rates} yields the desired estimate.
\end{proof}

The overall $L^1$-error is a direct consequence.

\begin{cor}[Convergence rate]\label{thm:convergenceRate}
Under the source conditions of \cref{thm:modificationError,thm:differenceToModification} any solution $u^\delta$ to $P_{\sqrt\delta}(f^\delta)$ satisfies
\begin{equation*}
\|u^\delta-u^\dagger\|_1
\leq C\delta^{1/4}
\end{equation*}
for a constant $C$ only depending on $u^\dagger$.
\end{cor}

\subsection{Convergence rates for level sets}

An $L^1$-convergence rate for piecewise constant functions immediately implies a corresponding convergence result for the superlevel sets.
Below, for a function $u:\Omega\to\R$ and $t\in\R$ we abbreviate the $t$-superlevel set as
\begin{equation*}
\{u\geq t\}=\{x\in\Omega\,|\,u(x)\geq t\}.
\end{equation*}
The symmetric difference between two sets $A,B\subset\Omega$ is denoted
\begin{equation*}
A\Delta B=(A\setminus B)\cup(B\setminus A),
\end{equation*}
and the Lebesgue measure of a measurable set $A\subset\Omega$ is indicated by $|A|$.

\begin{cor}[Convergence rates for level sets]\label{thm:levelsets}
Under the source conditions of \cref{thm:modificationError,thm:differenceToModification}, the $t$-superlevel sets of any solution $u^\delta$ to $P_{\sqrt\delta}(f^\delta)$ satisfy
\begin{equation*}
|\{u^\delta\geq t\}\Delta\{u^\dagger\geq t\}|\leq\frac{C\delta^{1/4}}{\dist(t,\mathrm{range}(u^\dagger))}
\end{equation*}
with a constant $C$ only depending on $u^\dagger$.
\end{cor}
\begin{proof}
Fix an arbitrary $s\notin\mathrm{range}(u^\dagger)$ and abbreviate $\varepsilon=\dist(s,\mathrm{range}(u^\dagger))$,
then $\{u^\dagger\geq t\}$ is independent of $t\in(s-\varepsilon,s+\varepsilon]$.
Now if
\begin{equation*}
|\{u^\delta\geq s\}\setminus\{u^\dagger\geq s\}|\geq|\{u^\dagger\geq s\}\setminus\{u^\delta\geq s\}|,
\end{equation*}
then due to the monotonicity of $\{u^\delta\geq t\}$ in $t$, for all $t\in]s-\varepsilon,s[$ we have
\begin{equation*}
|\{u^\delta\geq t\}\Delta\{u^\dagger\geq t\}|
\geq|\{u^\delta\geq t\}\setminus\{u^\dagger\geq t\}|
\geq|\{u^\delta\geq s\}\setminus\{u^\dagger\geq s\}|
\geq\tfrac12|\{u^\delta\geq s\}\Delta\{u^\dagger\geq s\}|.
\end{equation*}
Likewise, the opposite inequality at $s$ implies for all $t\in]s,s+\varepsilon[$
\begin{equation*}
|\{u^\delta\geq t\}\Delta\{u^\dagger\geq t\}|
\geq\tfrac12|\{u^\delta\geq s\}\Delta\{u^\dagger\geq s\}|.
\end{equation*}
Thus by \cref{thm:convergenceRate} and the area formula we have
\begin{multline*}
C\delta^{1/4}
\geq\|u^\delta-u^\dagger\|_1
=\int_{-\infty}^\infty|\{u^\delta\geq t\}\Delta\{u^\dagger\geq t\}|\wrt t\\
\geq\int_{s-\varepsilon}^{s+\varepsilon}|\{u^\delta\geq t\}\Delta\{u^\dagger\geq t\}|\wrt t
\geq\tfrac\varepsilon2|\{u^\delta\geq s\}\Delta\{u^\dagger\geq s\}|
\qedhere
\end{multline*}
\end{proof}

\section{Construction and estimation of dual variables}\label{sec:dualVariables}

The error estimates \cref{thm:exactRecoverySupport} to \cref{thm:levelsets} were all based on source conditions for the ground truth image $u^\dagger$.
In this section we will show that, as long as the minimum distance $\Delta$ between discontinuity lines of $u^\dagger$ is big enough,
these source conditions are all satisfied for the particular forward operator choice of the truncated Fourier transform \eqref{eqn:truncatedFourier}.
The source conditions are all of the form that a dual variable $w\in Y^\dual $ is sought such that $K^\dual w$ satisfies certain regularity conditions.
The dual operator $K^\dual $ is here given as
\begin{equation*}
K^\dual w=\sum_{|k|_\infty\leq\fcut}w_ke^{2\pi i(x,y)\cdot(k_1,k_2)}
\end{equation*}
so that the construction of dual variables reduces to the construction of trigonometric polynomials.
Because of their good localization properties, so-called Fej\'er kernels lend themselves as basic elements for such constructions.

\subsection{Properties of the Fej\'er kernel}
Let us abbreviate the so-called Dirichlet kernel \cite[\S\,15.2]{BrBrTh97} as
\begin{equation*}
\Dir_k(t)=\sum_{j=-k}^ke^{2\pi ijt}=\frac{\sin((2k+1)\pi t)}{\sin(\pi t)}
\end{equation*}
and the so-called Fej\'er kernel \cite[\S\,15.3]{BrBrTh97} as
\begin{equation*}
\Fej(t)
=\frac1{\fcut+1}\sum_{k=0}^{\fcut}\Dir_k(t)
=\sum_{k=-\fcut}^{\fcut}(1-\tfrac{|k|}{\fcut+1})\e^{2\pi ikt}
=\frac1{\fcut+1}\left(\frac{\sin((\fcut+1)\pi t)}{\sin(\pi t)}\right)^2
\end{equation*}
and recall that $\int_0^1\Fej(t)\wrt t=1$ and that $\Fej$ is nonnegative and periodic on $[0,1]$.
For notational simplicity we will in the following view the Fej\'er kernel as periodic function on $[-\frac12,\frac12]$.
Its function values and derivatives can be bounded as follows.

\begin{lem}[Estimates of Fej\'er kernel]\label{thm:FejerDerivatives}
There exists $C>0$ such that on $[-\frac12,\frac12]$ the derivatives of the Fej\'er kernel satisfy 
\begin{equation*}
|\Fej^{(j)}(t)|\leq C\min\left\{(\fcut+1)^{j+1},\tfrac{(\fcut+1)^{j-1}}{\sin^2(\pi t)}\right\}
\qquad\text{for }j=0,\ldots,4.
\end{equation*}
\end{lem}
\begin{proof}
By definition of the Fej\'er kernel we have
\begin{equation*}
|\Fej^{(j)}(t)|
=\frac1{\fcut+1}\left|\sum_{k=0}^{\fcut}\Dir_k^{(j)}(t)\right|
\leq\frac1{\fcut+1}\sum_{k=0}^{\fcut}\left|\Dir_k^{(j)}(t)\right|
\leq\max_{k=0,\ldots,\fcut}\left|\Dir_k^{(j)}(t)\right|.
\end{equation*}
The derivative of the Dirichlet kernel can readily be estimated as
\begin{equation*}
|\Dir_k^{(j)}(t)|
=\left|\frac{\d^j}{\d t^j}\sum_{l=-k}^{k}e^{2\pi ilt}\right|
=\left|\sum_{l=-k}^{k}(2\pi il)^je^{2\pi ilt}\right|
\leq(2k+1)(2\pi k)^j
\lesssim k^{j+1}
\end{equation*}
so that $|\Fej^{(j)}(t)|\lesssim\fcut^{j+1}$.
Now for $t<\frac1{\fcut}$ it holds
\begin{equation*}
\frac1{\sin^2(\pi t)}\geq\frac1{\sin^2(\pi/\fcut)}\gtrsim\fcut^2
\end{equation*}
so that $|\Fej^{(j)}(t)|\lesssim\frac{\fcut^{j-1}}{\sin^2(\pi t)}$ as desired.
For $t\geq\frac1{\fcut}$ we use induction in $j$.
For the zeroth derivative $j=0$ the statement is fulfilled by the explicit expression for $\Fej(t)$.
Now for complex trigonometric polynomials $T(t)=\sum_{l=-k}^ka_le^{i2\pi lt}$ of degree $k$, the Bernstein inequality reads
\begin{equation*}
\sup_{t\in\R}|T'(t)|\leq2\pi k\sup_{t\in\R}|T(t)|.
\end{equation*}
Thus we have
\begin{align*}
|\sin^2(\pi t)\Fej^{(j)}(t)|
&\leq|\underbrace{\sin^2(\pi t)\Fej^{(j)}(t)+2\sin(\pi t)\cos(\pi t)\Fej^{(j-1)}(t)}_{\frac\d{\d t}(\sin^2(\pi t)\Fej^{(j-1)}(t))}|+|2\sin(\pi t)\cos(\pi t)\Fej^{(j-1)}(t)|\\
&\leq|2\sin(\pi t)\cos(\pi t)\Fej^{(j-1)}(t)|+2\pi(\fcut+2)\sup_{t\in\R}|\sin^2(\pi t)\Fej^{(j-1)}(t)|.
\end{align*}
For $t\geq\frac1{\fcut}$ we have $|2\sin(\pi t)\cos(\pi t)\Fej^{(j-1)}(t)|\lesssim\fcut|\sin^2(\pi t)\Fej^{(j-1)}(t)|$ so that overall
\begin{equation*}
|\sin^2(\pi t)\Fej^{(j)}(t)|
\lesssim\fcut\sup_{t\in\R}|\sin^2(\pi t)\Fej^{(j-1)}(t)|
\lesssim\fcut^{j-1}
\end{equation*}
as desired.
\end{proof}

\subsection{One-dimensional trigonometric polynomials with good localization}

Given points $t_1,\ldots,t_I\in\R/\Z$ with minimum distance at least $\Delta$ and signs $s_1,\ldots,s_I\in\{-1,0,1\}$,
we now aim to construct a trigonometric polynomial $g$ on $\R/\Z$ with $|g|\leq1$ and $g(t_i)=s_i$ for all $i$.
Such a construction can already be found in \cite{Candes2014}, but for the sake of self-containedness we briefly provide it in this section
(note that the construction in \cite{Candes2014} is based on the squared Fej\'er kernel, while we employ the Fej\'er kernel itself, but this is merely a matter of taste).
To this end, as in \cite{Candes2014}, we consider the system of equations
\begin{equation}\label{eqn:trigConditions}
g(t_i)=s_i,\qquad
g'(t_i)=0,\qquad
i=1,\ldots,I.
\end{equation}
The Fej\'er kernel $\Fej$ has a pronounced maximum at $0$ and quickly decays to zero away from $0$ (for $k\to\infty$ it approximates the Dirac measure).
Hence a basic idea is to take $g$ as linear combination of the shifted kernels $\Fej(t-t_i)$.
Note, though, that, while $\Fej(t-t_i)$ has a pronounced maximum at $t_i$, the other summands (though small) may shift the extremum slightly away from $t_i$.
As a remedy, one could perturb $t_i$ to some $\tilde t_i$; however, finding the correct $\tilde t_i$ is a highly nonlinear problem.
Instead we exploit $\Fej(t-\tilde t_i)\approx \Fej(t-t_i)+(\tilde t_i-t_i)\Fej'(t-t_i)$ and thus take the ansatz
(analogously to \cite{Candes2014})
\begin{equation}\label{eqn:trigAnsatz}
g(t)=\sum_{i=1}^I\alpha_i\Fej(t-t_i)+\beta_i\Fej'(t-t_i).
\end{equation}

\begin{lem}[Fej\'er coefficients]\label{thm:FejerCoefficients}
Let $g$ from \eqref{eqn:trigAnsatz} satisfy \eqref{eqn:trigConditions}.
Then its coefficients $\alpha=(\alpha_1,\ldots,\alpha_I)^T$ and $\beta=(\beta_1,\ldots,\beta_I)^T$ satisfy
\begin{equation*}
\underbrace{%
\left(\begin{smallmatrix}
\frac{D_0}{\fcut+1}&\frac{D_1}{\sqrt{2/3}\pi(\fcut+1)^2}\\\frac{-D_1}{\sqrt{2/3}\pi(\fcut+1)^2}&\frac{-D_2}{\frac23\pi^2(\fcut+1)^3}
\end{smallmatrix}\right)
}_M
\underbrace{%
\left(\begin{smallmatrix}(\fcut+1)\alpha\\\sqrt{2/3}\pi(\fcut+1)^2\beta\end{smallmatrix}\right)
}_V
=\underbrace{%
\left(\begin{smallmatrix}s_1\\\vdots\\s_I\\0\\\vdots\\0\end{smallmatrix}\right)
}_W
\qquad\text{for }
D_j=(\Fej^{(j)}(t_l-t_i))_{l,i=1,\ldots,I}.
\end{equation*}
Moreover, there exists $\bar C>0$ such that $\Delta\geq\frac{\bar C}{\fcut+1}$ implies the solvability of the equation with
\begin{equation*}
\|V-W\|_\infty\leq\frac{\bar C}{\Delta^2(\fcut+1)^2}.
\end{equation*}
\end{lem}
\begin{proof}
First, it is straightforward to check that \eqref{eqn:trigConditions} is equivalent to
\begin{equation*}
\left(\begin{smallmatrix}
D_0&D_1\\D_1&D_2
\end{smallmatrix}\right)
\left(\begin{smallmatrix}\alpha\\\beta\end{smallmatrix}\right)
=W,
\end{equation*}
which in turn is equivalent to the given system $MV=W.$
Since $\Fej'$ is odd, $D_1$ is antisymmetric so that $M$ is actually symmetric.
We next show that $M$ is also diagonally dominant, so let $M_d$ be the diagonal and $\tilde M$ the rest of $M$,
\begin{equation*}
M=M_d+\tilde M
\qquad\text{with}\qquad
M_d=\left(\begin{smallmatrix}1\\&\ddots\\&&1\\&&&1-\frac1{(\fcut+1)^2}\\&&&&\ddots\\&&&&&1-\frac1{(\fcut+1)^2}\end{smallmatrix}\right).
\end{equation*}
We now estimate the row sum norm $\|\tilde M\|_\infty$ of $\tilde M$.
For $l\leq I$, using \cref{thm:FejerDerivatives} and $\sin^2(\pi t)\geq4t^2$ on $[-\frac12,\frac12]$ we have
\begin{multline*}
\sum_{j\neq l}|M_{lj}|
=\sum_{j\neq l,j\leq I}\left|\tfrac{(D_0)_{lj}}{\fcut+1}\right|
+\sum_{j=1}^I\left|\tfrac{(D_1)_{lj}}{\sqrt{2/3}\pi(\fcut+1)^2}\right|\\
\leq\sum_{j\neq l,j\leq I}\frac{C}{(\fcut+1)^24\dist(t_j,t_l)^2}
+\sum_{j\neq l,j\leq I}\frac{C\sqrt3}{\sqrt2\pi(\fcut+1)^24\dist(t_j,t_l)^2}.
\end{multline*}
Note that the distance of $t_l$ to the (left and right) nearest neighbours is at least $\dist(t_{l\pm1},t_l)\geq\Delta$.
Likewise, the distance to the two second-nearest neighbours must be at least $2\Delta$ and so on so that
\begin{equation*}
\sum_{j\neq l,j\leq I}\frac1{\dist(t_j,t_l)^2}
\leq2\sum_{i=1}^\infty\frac1{(i\Delta)^2}
=\frac{\pi^2}{3\Delta^2}.
\end{equation*}
Consequently, for some constant $\hat C>0$ we arrive at
\begin{align*}
\sum_{j\neq l}|M_{lj}|
&\leq\frac{\hat C}{(\fcut+1)^2\Delta^2}.
\end{align*}
Similarly, we can calculate
\begin{multline*}
\sum_{j\neq I+l}|M_{I+l,j}|
=\sum_{j=1}^{I}\left|\tfrac{(D_1)_{lj}}{\sqrt{2/3}\pi(\fcut+1)^2}\right|
+\sum_{j\neq l,j\leq I}\left|\tfrac{(D_2)_{lj}}{2\pi^2(\fcut+1)^3/3}\right|\\
\leq\sum_{j\neq l,j\leq I}\frac{C\sqrt3}{\sqrt2\pi(\fcut+1)^24\dist^2(t_j,t_l)}
+\sum_{j\neq l,j\leq I}\frac{3C}{2\pi^2(\fcut+1)^24\dist^2(t_j,t_l)}%
\leq\frac{\hat C}{(\fcut+1)^2\Delta^2}
\end{multline*}
if $\hat C$ was chosen large enough.
In summary, if $\Delta\geq\frac{2\sqrt{\hat C}}{\fcut+1}$, then $\|\tilde M\|_\infty\leq\frac14$ and thus $M$ is diagonally dominant.
As a direct consequence we obtain boundedness of $M^{-1}$,
\begin{equation*}
\|M^{-1}\|_\infty
\leq\left(\min_{i=1,\ldots,2I}(M_d)_{ii}-\|\tilde M\|_\infty\right)^{-1}
=\left(1-\frac1{(\fcut+1)^2}-\|\tilde M\|_\infty\right)^{-1}
\leq2,
\end{equation*}
so that the system of equations is uniquely solvable with
$\|V\|_\infty=\|M^{-1}W\|_\infty\leq\|M^{-1}\|_\infty\|W\|_\infty\leq2$.
From $M_dV-W=-\tilde MV$ we now obtain $V-W=V-M_d^{-1}W=-M_d^{-1}\tilde MV$ and thus
\begin{equation*}
\|V-W\|_\infty\leq\|M_d^{-1}\|_\infty\|\tilde M\|_\infty\|V\|_\infty\leq\frac{2\hat C}{(\fcut+1)^2\Delta^2}
\end{equation*}
as desired.
\end{proof}

Recall that $B_R(S)$ denotes the open $R$-neighbourhood of the set $S$.
Below we abbreviate $B_i=B_R(\{t_i\})$, $B=\bigcup_{i=1}^NB_i$, and $B^c=(\R/\Z)\setminus B$.
The estimates to be derived are illustrated in \cref{fig:trigPolynomialBounds}.

\begin{prop}[Existence of trigonometric polynomial]\label{thm:existencePolynomials}
There exist constants $C_1,C_2,C_3,C_4>0$
such that if $\Delta\geq\frac{C_1}{\fcut+1}$,
then there exists a unique $g$ of the form \eqref{eqn:trigAnsatz} satisfying
\begin{equation*}
|g|\leq1-\kappa R^2\text{ on }B^c,\qquad
s_ig(t)\leq1-\kappa\dist(t,t_i)^2\text{ on }B_i,\qquad
|g(t)-s_i|\leq\eta\dist(t,t_i)^2\text{ on }B_i
\end{equation*}
with $R=\frac{C_2}{\fcut+1}$, $\eta=C_3(\fcut+1)^2$, $\kappa=C_4(\fcut+1)^2$.
The last inequality implies in particular \eqref{eqn:trigConditions}.
\end{prop}
\begin{proof}
Take $g$ to be the unique solution to \eqref{eqn:trigConditions} from \cref{thm:FejerCoefficients}
and recall that its coefficients satisfy $|\alpha_i|\lesssim(\fcut+1)^{-1}$ and $|\beta_i|\lesssim(\fcut+1)^{-2}$
if $\Delta\geq\frac{C_1}{\fcut+1}$ for some $C_1>0$ chosen large enough.
Now consider an arbitrary $t\in[-\frac12,\frac12]$ and let $t_l$ be closest to $t$, then using \cref{thm:FejerDerivatives} we can calculate
\begin{align*}
|g''(t)|
&\leq\sum_{i=1}^I|\alpha_i||\Fej''(t-t_i)|+|\beta_i||\Fej'''(t-t_i)|\\
&=|\alpha_l||\Fej''(t-t_l)|+|\beta_l||\Fej'''(t-t_l)|+\sum_{i\neq l}|\alpha_i||\Fej''(t-t_i)|+|\beta_i||\Fej'''(t-t_i)|\\
&\lesssim(\fcut+1)^2+\sum_{i\neq l}\frac1{\sin^2(\pi(t-t_i))}.
\end{align*}
As in the proof of \cref{thm:FejerCoefficients} we can now estimate
\begin{equation*}
\sum_{i\neq l}\frac1{\sin^2(\pi(t-t_i))}
\lesssim\sum_{i\neq l}\frac1{\dist(t,t_i)}
\lesssim\sum_{j=1}^\infty\frac1{(j\Delta)^2}
\lesssim\Delta^{-2}
\end{equation*}
so that for some constant $\hat C>0$ we have
\begin{equation*}
|g''(t)|
\leq\hat C\left((\fcut+1)^2+\tfrac1{\Delta^2}\right)
\leq\hat C\left(1+\tfrac1{C_1^2}\right)(\fcut+1)^2.
\end{equation*}
Consequently we can choose $C_3=\tfrac{\hat C}2\left(1+\tfrac1{C_1^2}\right)$ in order to satisfy the inequality involving $\eta$.
Analogously one obtains (after potentially increasing $\hat C$)
\begin{equation*}
|g'''(t)|
\leq\hat C\left(1+\tfrac1{C_1^2}\right)(\fcut+1)^3,
\end{equation*}
which we will need next to estimate $g''$ from below in $B_l$ in order to derive the inequality involving $\kappa$.
To this end assume $s_l\neq0$ (else the inequality involving $\kappa$ would already follow from the one involving $\eta$ if $C_2$ and thus $R$ is chosen small enough).
Using $|(\fcut+1)\alpha_l-s_l|\leq\frac{\bar C}{\Delta^2(\fcut+1)^2}\leq\frac{\bar C}{C_1^2}$ from \cref{thm:FejerCoefficients} we then estimate
\begin{multline*}
-s_lg''(t_l)
\geq-s_l\alpha_l\Fej''(0)-\sum_{i\neq l}(|\alpha_i||\Fej''(t_l-t_i)|+|\beta_i||\Fej'''(t_l-t_i)|)\\
\geq\left(1-\tfrac{\bar C}{C_1^2}\right)\tfrac23\pi^2\fcut(\fcut+2)-\sum_{i\neq l}\tfrac{\tilde C}{\sin^2(\pi(t-t_i))}
\geq\left(1-\tfrac{\bar C}{C_1^2}\right)(\fcut+1)^2-\tfrac{\check C}{\Delta^2}
\geq\left(1-\tfrac{\bar C+\check C}{C_1^2}\right)(\fcut+1)^2
\end{multline*}
for some constants $\tilde C,\check C>0$. Therefore, for any $r$ with $\dist(r,t_l)\leq\frac{C_2}{\fcut+1}$ with a constant $C_2>0$ to be fixed below, we obtain
\begin{multline*}
-s_lg''(r)
\geq\left(1-\tfrac{\bar C+\check C}{C_1^2}\right)(\fcut+1)^2-\max_{s\in[-\frac12,\frac12]}|g'''(s)|\dist(s,t_l)\\
\geq\left(1-\tfrac{\bar C+\check C}{C_1^2}\right)(\fcut+1)^2-\hat C\left(1+\tfrac1{C_1^2}\right)(\fcut+1)^3\dist(s,t_l)
\geq2C_4(\fcut+1)^2,
\end{multline*}
where $C_4=\left(1-\tfrac{\bar C+\check C}{C_1^2}\right)/2-C_2\hat C\left(1+\tfrac1{C_1^2}\right)/2$ is strictly positive
if $C_1$ is chosen large and $C_2$ small enough.
For this choice of $C_1,C_2,C_4$ and $R=\frac{C_2}{\fcut+1}$, $\kappa=C_4(\fcut+1)^2$ the inequality involving $\kappa$ then holds true.
Finally let us ensure the inequality on $B^c$.
To this end assume $t\in B^c$ and again that $t_l$ is closest to $t$.
Furthermore, we may without loss of generality assume $R<\Delta/2$, $\Fej(R)\leq \Fej(0)+\Fej''(0)R^2$, and $R<\frac1{2(\fcut+1)}$ so that $\Fej<\Fej(R)$ outside $[-R,R]$
(else we just sufficiently decrease $C_2$ further).
Furthermore, we recall from \cref{thm:FejerCoefficients} that $|\beta_l|\leq\frac{\bar C}{C_1^2}$.
Analogously to the previous calculation we can now estimate
\begin{align*}
|g(t)|
&\leq\sum_{i=1}^I|\alpha_i|\Fej(t-t_i)+|\beta_i||\Fej'(t-t_i)|\\
&\leq|\alpha_l|\Fej(R)+|\beta_l||\Fej'(t-t_l)|+\sum_{i\neq l}|\alpha_i|\Fej(t-t_i)+|\beta_i||\Fej'(t-t_i)|\\
&\leq(1+\tfrac{\bar C}{C_1^2})\frac{\Fej(R)}{\fcut+1}+|\beta_l||\Fej'(t-t_l)|+\sum_{i\neq l}\frac{\tilde C}{(\fcut+1)^2\sin^2(\pi(t-t_i))}\\
&\leq(1+\tfrac{\bar C}{C_1^2})\frac{\Fej(R)}{\fcut+1}+\frac{C\bar C}{C_1^2\sin^2(\pi R)}+\frac{\check C}{(\fcut+1)^2\Delta^2}\\
&\leq(1+\tfrac{\bar C}{C_1^2})\frac{\Fej(0)+\Fej''(0)R^2}{\fcut+1}+\frac{\underline C}{C_1^2}\\
&\leq(1+\tfrac{\bar C}{C_1^2})(1-\tfrac23\pi^2\fcut(\fcut+2)R^2)+\frac{\underline C}{C_1^2}
\end{align*}
for constants $\tilde C,\check C,\underline C>0$.
As required, this is smaller than $1-\kappa R^2$ if $C_4$ is chosen small and $C_1$ large enough.
\end{proof}

\begin{figure}
\centering
\includegraphics{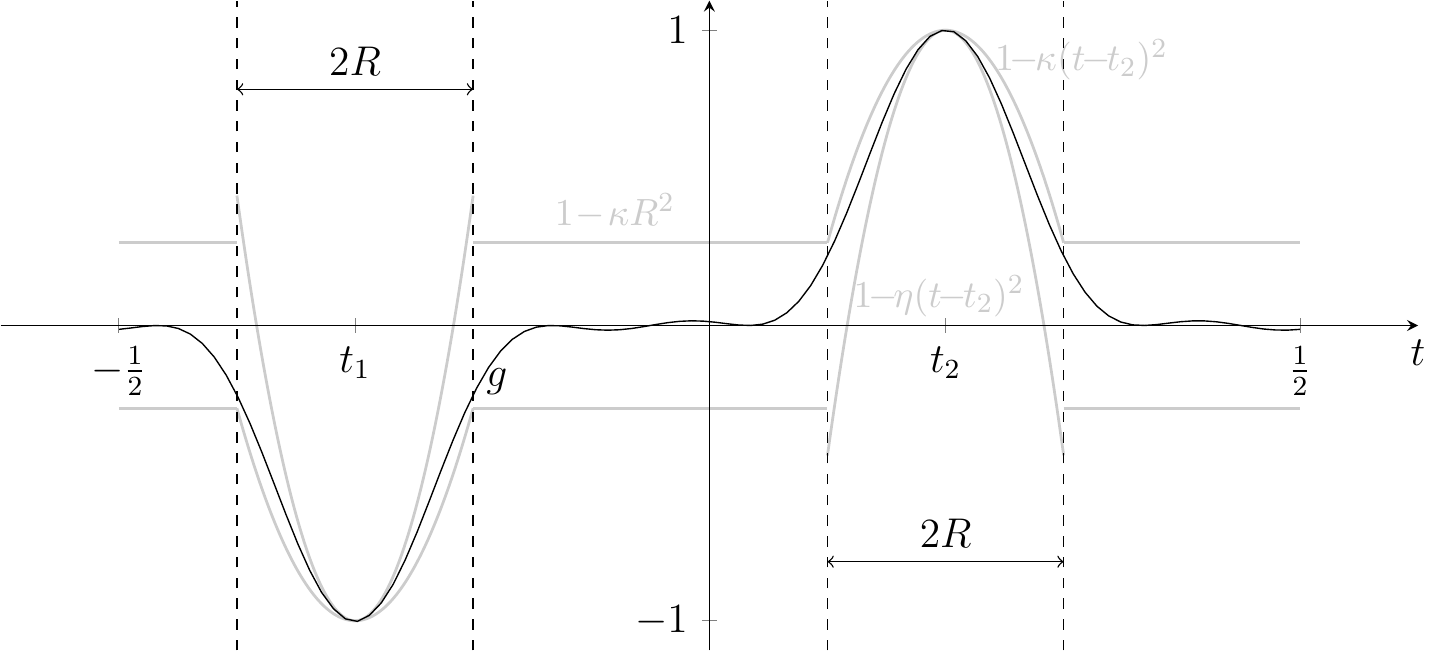}
\caption{Illustration of the bounds obtained in \cref{thm:existencePolynomials}.}
\label{fig:trigPolynomialBounds}
\end{figure}

\subsection{Approximation of characteristic functions with the Fej\'er kernel}
Given $a\in[0,1]$, we approximate the characteristic function $\chi_{[0,a]}:\R/\Z\to\R$ by
\begin{equation*}
\cha_{[0,a]}(t)
=\int_0^a\Fej(t-x)\wrt x.
\end{equation*}
An approximation of the characteristic function $\chi_{[a,b]}$ for $a,b\in\R/\Z$ is then given by
\begin{equation*}
\cha_{[a,b]}(t)=\cha_{[0,\length([a,b])]}(t-a).
\end{equation*}
Note that by construction we have $0\leq\cha_{[a,b]}\leq1$ as well as $\cha_{[a,b]^c}=1-\cha_{[a,b]}$.

\begin{lem}[Properties of approximate characteristic function]\label{thm:charFun}
\notinclude{
The approximate characteristic function satisfies
\begin{equation*}
\max\{0,1-\tfrac{3\pi}{2(\fcut+1)\dist(t,[a,b]^c)}\}
\leq \cha_{[a,b]}(t)
\leq\min\{1,\tfrac{3\pi}{2(\fcut+1)\dist(t,[a,b])}\}
\end{equation*}
as well as
\begin{equation*}
\length([a,b])>\int_a^b\cha_{[a,b]}(t)\wrt t>\tfrac{13}{20}\length([a,b])
\qquad\text{as long as }\length([a,b])\geq\frac1{\fcut+1}.
\end{equation*}
\todo{first part does not seem to be used...}
}%
As soon as $l_a^b=\length([a,b])>\frac1{\fcut+1}$, the approximate characteristic function $\cha_{[a,b]}$ satisfies
\begin{equation*}
l_a^b>\int_a^b\cha_{[a,b]}(t)\wrt t>\max\left\{\frac{13}{20},1-\frac{5+\log(2\pi(\fcut+1)l_a^b)}{\pi^2(\fcut+1)l_a^b}\right\}l_a^b.
\end{equation*}
\end{lem}
\begin{proof}
It suffices to consider the case $\cha_{[0,a]}$ with $a\in[0,1]$.
\notinclude{
Furthermore, for $t\in[\frac a4,\frac a2]$ (i.e. at the center of the approximated characteristic function and away from the interface) we write
\begin{multline*}
\cha_{[0,a]}(t)-\cha_{[0,a]}(\tfrac a2)
=\int_0^a\Fej(t-x)-\Fej(\tfrac a2-x)\wrt x
=\int_{t-a}^t\Fej(y)\wrt y-\int_{\frac a2-a}^{\frac a2}\Fej(y)\wrt y\\
=\int_{[t-a,-\frac a2]}\Fej(y)\wrt y-\int_{[t,\frac a2]}\Fej(y)\wrt y
=\int_{[\frac a2,a-t]}\Fej(y)\wrt y-\int_{[t,\frac a2]}\Fej(y)\wrt y
\end{multline*}
using the symmetry of $\Fej$ in the last step.
The Hessian is then readily computed as
\begin{equation*}
(\cha_{[0,a]}(t)-\cha_{[0,a]}(\tfrac a2))''
=\Fej'(t)+\Fej'(a-t),
\end{equation*}
while $(\cha_{[0,a]}(t)-\cha_{[0,a]}(\tfrac a2))'|_{t=\frac a2}=0$ due to symmetry.
Hence we get
\begin{equation*}
|\cha_{[0,a]}(t)-\cha_{[0,a]}(\tfrac a2)|
\leq\max_{\tau\in[t,\frac a2]}\frac12|\Fej'(\tau)+\Fej'(a-\tau)|(t-\tfrac a2)^2
\leq\max_{\tau\in[\frac a4,\frac a2]}|\Fej'(\tau)|(t-\tfrac a2)^2.
\end{equation*}
Using $\Fej'(t)=\frac{\pi\sin(\pi(\fcut+1)t)}{(\fcut+1)\sin^3(\pi t)}(\fcut\sin(\pi(\fcut+2)t)-(\fcut+2)\sin(\pi\fcut t))$ we see $|\Fej'(t)|\leq\frac{2\pi}{|\sin(\pi t)|^3}$ and thus
\begin{equation*}
|\cha_{[0,a]}(t)-\cha_{[0,a]}(\tfrac a2)|
\leq\frac{2\pi}{|\sin(\pi\frac a4)|^3}(t-\tfrac a2)^2.
\end{equation*}
}%
The left inequality follows from $\cha_{[0,a]}\leq1$.
For the right one we recall the definition of the sine and cosine integrals
\begin{equation*}
\Si(s)=\int_0^s\frac{\sin t}t\wrt t
\qquad\text{and}\qquad
\Cin(s)=\int_0^s\frac{1-\cos t}t\wrt t
\end{equation*}
(note that the more common cosine integral $\Ci$ is defined as $\Ci(s)\notinclude{=-\int_s^\infty\frac{1-\cos t}t\wrt t}=\gamma+\log(s)-\Cin(s)$ for $\gamma$ the Euler--Mascheroni constant) and calculate
\begin{align*}
\int_0^a\cha_{[0,a]}(t)\wrt t
&=\int_0^a\int_0^a\Fej(t-x)\wrt x\wrt t
\geq\int_0^a\int_0^a\frac1{\fcut+1}\left(\frac{\sin((\fcut+1)\pi(t-x))}{\pi(t-x)}\right)^2\wrt x\wrt t\\
\notinclude{
&=\frac{\Si(2\pi(\fcut+1)x)x+\Si(2\pi(\fcut+1)t)t-\Si(2\pi(\fcut+1)(x-t))(x-t)}\pi\\
&\quad+\frac1{2\pi^2(\fcut+1)}\bigg(\cos(2\pi(\fcut+1)x)+\cos(2\pi(\fcut+1)t)-\cos(2\pi(\fcut+1)(x-t))-1-\gamma\\
&\quad+\Ci(2\pi(\fcut+1)x)+\Ci(2\pi(\fcut+1)t)-\Ci(2\pi(\fcut+1)(x-t))+\log\frac{x-t}{2\pi(\fcut+1)xt}\bigg)\bigg|_{\substack{x=a,\\t=a}}\\
}%
&=\frac{\Si(2\pi(\fcut+1)x)x+\Si(2\pi(\fcut+1)t)t-\Si(2\pi(\fcut+1)(x-t))(x-t)}\pi\\
&\quad+\frac1{2\pi^2(\fcut+1)}\bigg(\cos(2\pi(\fcut+1)x)+\cos(2\pi(\fcut+1)t)-\cos(2\pi(\fcut+1)(x-t))-1\\
&\quad-\Cin(2\pi(\fcut+1)x)-\Cin(2\pi(\fcut+1)t)+\Cin(2\pi(\fcut+1)(x-t))\bigg)\bigg|_{x=0}^a\bigg|_{t=0}^a\\
&=\frac{2a}\pi f(2\pi(\fcut+1)a)\\
\text{for }f(r)&=\Si(r)+\frac{\cos(r)-\Cin(r)-1}{r}.
\end{align*}
The function $f$ is monotonically increasing on the positive halfline due to $f'(r)=\Cin(r)/r^2>0$.
Thus, for $a\geq\frac1{\fcut+1}$ we get
$
\int_0^a\cha_{[0,a]}(t)\wrt t
\geq\frac{2a}\pi f(2\pi)
>\frac{13a}{20}.
$
In addition, $f(r)$ converges to $\frac\pi2$ as $r\to\infty$ with
\begin{equation*}
0\leq\frac\pi2-f(r)
=\left(\frac\pi2-\Si(r)\right)+\frac{1-\cos(r)+\gamma-\Ci(r)+\log(r)}r
\leq\frac1r+\frac{3+\gamma+\log(r)}r
\end{equation*}
for $r\geq1$, which finally yields the desired estimate.
\notinclude{
For $a\geq\frac12$ we then simply estimate
\begin{multline*}
\int_0^a\cha_{[0,a]}(t)\wrt t
=\int_0^a\cha_{[0,a/2]}(t)\wrt t+\int_0^a\cha_{[a/2,a]}(t)\wrt t\\
\geq\int_0^{\frac a2}\cha_{[0,a/2]}(t)\wrt t+\int_{\frac a2}^a\cha_{[a/2,a]}(t)\wrt t
=2\int_0^{\frac a2}\cha_{[0,a/2]}(t)\wrt t
>\frac{13a}{20}.
\qedhere
\end{multline*}
}%
\end{proof}

\notinclude{
The bounds on $\cha_{[a,b]}$ are illustrated in \cref{fig:DualFcn1D}.

\begin{figure}
\centering
\setlength\unitlength{.2\linewidth}
\begin{picture}(2,1)
\put(0,0){\includegraphics[width=2\unitlength,height=\unitlength,trim=66 50 10 10,clip]{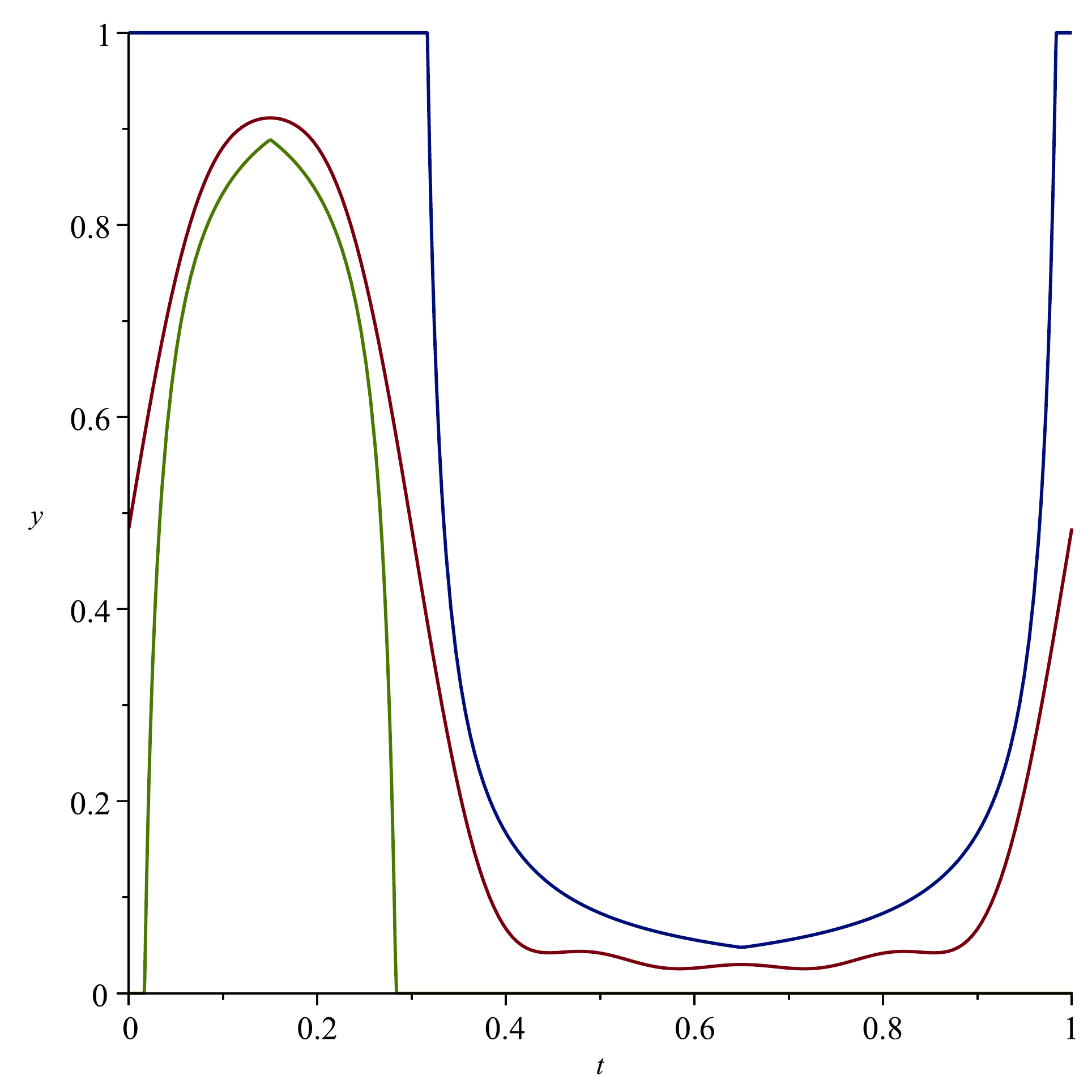}}
\put(0,0){\vector(1,0){2.1}}
\put(0,0){\vector(0,1){1.1}}
\put(-.03,-.1){0}
\put(1.98,-.1){1}
\put(-.1,-.04){0}
\put(-.1,.96){1}
\end{picture}
\caption{The function $\cha_{[0,a]}$ for $\fcut=5$ and $a=\frac3{10}$ and a sketch of its upper and lower bounds.}
\label{fig:DualFcn1D}
\end{figure}
}%

\subsection{Construction of dual certificates}\label{sec:dualCertificates}

After the previous preparations, we are now in a position to prove that all source conditions in \crefrange{thm:exactRecoverySupport}{thm:modificationError} are satisfied
as long as $\Delta$ is sufficiently large.

\begin{prop}[Existence of dual certificates I]\label{thm:dualCertificatesI}
There exist constants $C_1,C_2,C_4>0$ such that choosing $R=\frac{C_2}{\fcut+1}$ and $\kappa=C_4(\fcut+1)^2$
the dual variables $w^\hor,w^\ver$ from \cref{thm:exactRecoverySupport,thm:TVestimates} exist whenever $\Delta>\frac{C_1}{\fcut+1}$.
\end{prop}
\begin{proof}
Just define $h^\ver$ to be the trigonometric polynomial from \cref{thm:existencePolynomials} for $t_i=x_i$ and $s_i=s_i^\ver$
and analogously $h^\hor$ to be the trigonometric polynomial for $t_i=y_i$ and $s_i=s_i^\hor$.
Then set $g^\ver(x,y)=h^\ver(x)$ and $g^\hor(x,y)=h^\hor(y)$.
Since $g^\hor,g^\ver$ are trigonometric polynomials of degree $\fcut$, so are $\D_yg^\hor$ and $\D_xg^\ver$, which thereby lie in the range of $K^\dual $.
Thus there exist $w^\hor,w^\ver$ with $K^\dual w^\hor=\D_yg^\hor$ and $K^\dual w^\ver=\D_xg^\ver$, as desired.
\end{proof}

\begin{prop}[Existence of dual certificates II]\label{thm:dualCertificatesII}
The dual variable $w$ from \cref{thm:exactRecovery} exists whenever $\Delta\geq\frac{3}{\fcut+1}$.
\end{prop}
\begin{proof}
We simply take
\begin{equation*}
g(x,y)=\sum_{m,n}s_{mn}\cha_{[x_m,x_{m+1}]}(x)\cha_{[y_n,y_{n+1}]}(y),
\end{equation*}
then on $[x_{\bar m},x_{\bar m+1}]\times[y_{\bar n},y_{\bar n+1}]$ we have
$g(x,y)=s_{\bar m\bar n}\cha_{[x_{\bar m},x_{\bar m+1}]}(x)\cha_{[y_{\bar n},y_{\bar n+1}]}(y)+R(x,y)$
with the remainder
\begin{multline*}
|R(x,y)|
=\left|\sum_{(m,n)\neq(\bar m,\bar n)}s_{mn}\cha_{[x_m,x_{m+1}]}(x)\cha_{[y_n,y_{n+1}]}(y)\right|\\
\leq\sum_{(m,n)\neq(\bar m,\bar n)}\cha_{[x_m,x_{m+1}]}(x)\cha_{[y_n,y_{n+1}]}(y)
=1-\cha_{[x_{\bar m},x_{\bar m+1}]}(x)\cha_{[y_{\bar n},y_{\bar n+1}]}(y).
\end{multline*}
With this and the estimate $$\int_a^b\cha_{[a,b]}(t)\wrt t>\left(1-\frac{5+\log(2\pi(\fcut+1)l_a^b)}{\pi^2(\fcut+1)l_a^b}\right)l_a^b>\left(1-\frac{5+\log(6\pi)}{3\pi^2}\right)l_a^b>0.73l_a^b$$ from \cref{thm:charFun} for $l_a^b=\length([a,b])\geq\Delta>\frac3{\fcut+1}$ we can estimate
\begin{align*}
s_{\bar m\bar n}\int_{x_{\bar m}}^{x_{\bar m+1}}\int_{y_{\bar n}}^{y_{\bar n+1}}g(x,y)\wrt(x,y)
&=\int_{x_{\bar m}}^{x_{\bar m+1}}\int_{y_{\bar n}}^{y_{\bar n+1}}\cha_{[x_{\bar m},x_{\bar m+1}]}(x)\cha_{[y_{\bar n},y_{\bar n+1}]}(y)+s_{\bar m\bar n}R(x,y)\wrt(x,y)\\
&\geq\int_{x_{\bar m}}^{x_{\bar m+1}}\int_{y_{\bar n}}^{y_{\bar n+1}}2\cha_{[x_{\bar m},x_{\bar m+1}]}(x)\cha_{[y_{\bar n},y_{\bar n+1}]}(y)-1\wrt(x,y)\\
&\geq(2\cdot0.73^2-1)\length([x_{\bar m},x_{\bar m+1}])\length([y_{\bar n},y_{\bar n+1}])
>0.
\qedhere
\end{align*}
\end{proof}

Finally we aim to obtain a dual variable $g$ for the convergence rate result.

\begin{prop}[Existence of dual certificates III]\label{thm:dualCertificatesIII}
There exist constants $C_1,C_2,C_3>0$ such that if $\Delta\geq\frac{C_1}{\fcut+1}$, then the dual certificates $v^\hor,v^\ver$ from \cref{thm:modificationError} exist for $R=\frac{C_2}{\fcut+1}$, $\eta=C_3(\fcut+1)^2$ and $C=\frac3{10}$.
\end{prop}
\begin{proof}
We just consider $v^\ver$, the existence of $v^\hor$ follows analogously.
We will actually construct the corresponding function $g^\ver$ from \cref{thm:modificationError} as a trigonometric polynomial of degree $\fcut$,
then $-\D_xg^\ver$ is a trigonometric polynomial of the same degree and hence in the range of $K^\dual $ so that there exists the desired $v^\ver$ with $K^\dual v^\ver=-\D_xg^\ver$.

To this end let $s_{mn}^\ver\in\{-1,1\}$ for $m=1,\ldots,M$, $n=1,\ldots,N$ be given.
We take $C_1,C_2,C_3$ to be the values from \cref{thm:existencePolynomials} and set
\begin{equation*}
g^\ver(x,y)=\sum_{m=1}^Mg_m(x)\sum_{n=1}^Ns_{kl}\cha_{[y_n,y_{l+1}]}(y),
\end{equation*}
where $g_m$ is the function from \cref{thm:existencePolynomials} for $t_i=x_i$ and $s_m=1$ and $s_i=0$ else.
By construction we have $g^\ver\in\mathrm{range}K^\dual $ and
\begin{equation*}
|g^\ver(x,y)|\leq\sum_{m=1}^Mg_m(x).%
\end{equation*}
Now note that the right-hand side equals the function constructed in \cref{thm:FejerCoefficients,thm:existencePolynomials} for the signs $s_i=1$.
By \cref{thm:existencePolynomials} this function has absolute value bounded by one, thus $|g^\ver(x,y)|\leq1$.
Then for any $x\in[x_{\bar m}-R,x_{\bar m}+R]$ we have
\begin{equation*}
|g^\ver(x,y)-g^\ver(x_{\bar m},y)|
=\left|\sum_{m=1}^M(g_m(x)-g_m(x_{\bar m}))\sum_{n=1}^Ns_{kl}\cha_{[y_n,y_{l+1}]}(y)\right|
\leq\sum_{m=1}^M|g_m(x)-g_m(x_{\bar m})|.
\end{equation*}
Now let $s_m$ be the sign of $g_m(x)-g_m(x_{\bar m})$, then
\begin{equation*}
\sum_{m=1}^M|g_m(x)-g_m(x_{\bar m})|
=\sum_{m=1}^Ms_m(g_m(x)-g_m(x_{\bar m}))
=\tilde g(x)-\tilde g(x_{\bar m}),
\end{equation*}
where $\tilde g$ is the function constructed in \cref{thm:FejerCoefficients,thm:existencePolynomials} for the signs $s_m$.
By \cref{thm:existencePolynomials} we thus have
\begin{equation*}
|g^\ver(x,y)-g^\ver(x_{\bar m},y)|
\leq\eta\dist(x,x_{\bar m})^2.
\end{equation*}
Furthermore, on $\{x_{\bar m}\}\times[y_{\bar n},y_{\bar n+1}]$ we have
\begin{equation*}
g^\ver(x_{\bar m},y)=\sum_{n=1}^Ns_{mn}^\ver\cha_{[y_n,y_{n+1}]}(y)
=s_{\bar m\bar n}^\ver\cha_{[y_{\bar n},y_{\bar n+1}]}(y)+R_{\bar m\bar n}(y),
\end{equation*}
whose remainder can be estimated by
\begin{equation*}
|R_{\bar m\bar n}(y)|
=\left|\sum_{n\neq\bar n}s_{\bar mn}^\ver\cha_{[y_n,y_{n+1}]}(y)\right|
\leq\sum_{n\neq\bar n}\cha_{[y_n,y_{n+1}]}(y)
=\cha_{[y_{\bar n},y_{\bar n+1}]^c}(y)
=1-\cha_{[y_{\bar n},y_{\bar n+1}]}(y).
\end{equation*}
Thus, using \cref{thm:charFun} we obtain
\begin{align*}
s_{mn}^\ver\int_{y_{n}}^{y_{n+1}}g^\ver(x_{m},y)\wrt y
&=\int_{y_{n}}^{y_{n+1}}\cha_{[y_{n},y_{n+1}]}(y)+s_{mn}^\ver R_{mn}(y)\wrt y\\
&\geq\int_{y_n}^{y_{n+1}}(2\cha_{[y_{n},y_{n+1}]}(y)-1)\wrt y\\
&\geq(2\tfrac{13}{20}-1)\length([y_n,y_{n+1}]).
\qedhere
\end{align*}
\end{proof}

\section{Numerical experiments}
In this section, we provide numerical experiments that aim at confirming the exact recovery and convergence results of \cref{thm:exact_recon_intro,thm:convergence_intro} numerically.
\subsection{Data generation and algorithmic setup}
\paragraph*{Data generation.}
The data for the numerical experiments was created as follows. First, given an array of bins $(I^\Delta_1,\ldots,I^\Delta_K)$ for the minimal jump-point distances (here we used $I^\Delta_k = [k\cdot 0.01 - 0.005,k\cdot0.01+0.005]$ with $k=1,\ldots,10$), for each $I^\Delta_k$, jump points $\{(x_m)_{m=1}^M,(y_n)_{n=1}^N\}$ were created by first randomly selecting $N,M \in \{2,\ldots,20\}$ and then choosing the points $(x_m)_{m=1}^M$, $(y_n)_{n=1}^N$ in $[0,1[$ uniformly at random, where the interval $[0,1[$ was discretized using $120$ points.
The jump points were accepted if the minimum distance $\Delta$ of the points was in $I^\Delta_k$. The process was repeated until a total of 100 different sets of jump points were found for each bin $I^\Delta_k$.

Then, for each set of jump points $(x_m)_{m=1}^M$, $(y_n)_{n=1}^N$ of each bin $I^\Delta_k$, the values $u^\dagger_{mn} $ of the image
\begin{equation*}
u^\dagger = \sum_{m=1}^M \sum_{n=1}^N u^\dagger_{mn} \chi_{[x_m,x_{m+1}[ \times [y_n,y_{n+1}[}\in\BV(\Omega)
\end{equation*}
were defined by the following greedy strategy such that \cref{ass:consistentGradient} is fulfilled. First, an order for defining the values $(u^\dagger _{mn})_{mn}$ consecutively was selected at random. Then, for defining the value $u^\dagger _{mn}$ at any stage, again in random order, all horizontal and vertical neighbouring positions $(n+1,m),(n-1,m),(n,m+1),(n,m-1)$ were considered. If the value at a neighbouring position was already set and if the sign of the gradient corresponding to the edge between the two positions was already fixed, the range of possible values for $u^\dagger _{mn}$ was reduced such that, for any value in that range, \cref{ass:consistentGradient} was satisfied. In case this resulted in a degenerate range with no admissible values, the value at the neighbouring position together with other already specified values were corrected in such a way that (i) \cref{ass:consistentGradient} did not get violated at any other edge position and (ii) the range of possible values for $u^\dagger_{mn}$ was no longer degenerate. Then, given a non-degenerate range of admissible values for $u^\dagger_{mn}$, the value $u^\dagger_{mn}$ was selected uniformly at random in that range. Further, if the newly specified value $u^\dagger_{mn}$ determined a sign of the grey value jump across one of its edges, then the same jump sign was assigned to all edges on the same horizontal or vertical line.%
Two examples of images $u^\dagger $ created by this procedure can be found in \cref{fig:image_examples}.

\begin{figure}[t]
\centering
\includegraphics[width=0.25\linewidth]{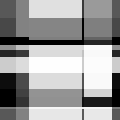}
\hspace*{1cm}
\includegraphics[width=0.25\linewidth]{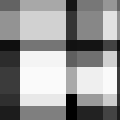}
\caption{\label{fig:image_examples} Two example images of $\Dvalid$ for $I^\delta = [0.015,0.025]$ (left) and $I^\delta = [0.085,0.095]$ (right)}
\end{figure}

The above procedure resulted in a data set of $1000$ admissible images $\Dvalid$, containing, for each bin $I^\Delta_k$ with $k=1,\ldots,10$, a set of $100$ different images with minimum jump-point distance $\Delta$ in the bin $I^\Delta_k$.
In order to test the necessity of \cref{ass:consistentGradient}, a second data set $\Dinvalid$ was generated by flipping exactly two randomly chosen neighbouring values for each image in $\Dvalid$ in such a way that \cref{ass:consistentGradient} did no longer hold.

The two data sets $\Dvalid$ and $\Dinvalid$ were used to numerically test the exact reconstruction result of \cref{thm:exact_recon_intro}. 
In addition, a data set $\Dconv$ was defined for numerically evaluating the convergence result of \cref{thm:convergence_intro}. The latter was obtained by selecting at most 5 images for each bin $I^\Delta_k$ of the data set $\Dvalid$, where the reconstruction with a frequency cutoff $\fcut=18$ for the forward operator $K$ (see \cref{eqn:truncatedFourier}) was successful. This resulted in 5 images for each bin $I^\Delta_k$ for $k=2,\ldots,10$ and in one image for the bin $I^\Delta _1$, such that in total, $\Dconv$ consists of $1 + 9\cdot 5 = 46$ images.

In a addition, also a real test image from a test image dataset (the \emph{Cameraman} image) was used in the numerical experiments to visualize reconstruction results with increasing frequency cutoff. 

\paragraph*{Algorithmic setup.} To numerically solve \ref{eq:main_min_prob} both for $\alpha = 0 $ and $\alpha>0$, the primal-dual algorithm of \cite{Pock11_primal_dual} was applied to an appropriate saddle-point reformulation (with the regularization functional being dualized) and implemented with a GPU-based parallelization using Python with PyOpenCL \cite{klockner2012pycuda}. Choosing the stepsize according to \cite{Pock11_primal_dual}, convergence of the algorithm to a global optimum of \ref{eq:main_min_prob} can be guaranteed. For details regarding the implementation we refer to the publicly available source code \cite{paper_source_code}, which in particular contains scripts to reproduce all experiments shown here. We highlight that the algorithm was implemented in such a way that the unknown images are constrained to be real at any point during the iterations and such that the proximal mapping (see \cite{Pock11_primal_dual}) of the data term is explicit both for $\alpha=0$ and $\alpha>0$. In case of noisy data, the data was symmetrized in a preprocessing step such that it corresponds to the Fourier transform of an image with only real-valued entries.

\subsection{Experimental results}
\paragraph*{Experiments regarding exact reconstruction.} To evaluate the exact reconstruction result of \cref{thm:exact_recon_intro} numerically, \ref{eq:main_min_prob} was solved for all images in $\Dvalid$ and $\Dinvalid$, once using a frequency cutoff of $\fcut=12$ and once using a cutoff of $\fcut=18$. Results showing both the $L^1$-distance to the ground truth and the percentage of correctly reconstructed images are provided in \cref{fig:error_plots}. Here, a reconstruction was considered correct if both the gradient of the reconstructed image coincided with the gradient of the ground truth pointwise up to a tolerance of $10^{-4}$ and if the values of the reconstructed and ground truth images at all points $(x_m,y_n)$, $m=1,\ldots,M$, $n=1,\ldots,N$ coincided up to the same tolerance. The red bars show the non-symmetric mean deviation of the $L^1$ error from its mean, i.e., the the upper bar shows the mean deviation of all values above the mean and the lower bar shows the mean deviation of all values below the mean (note also the logarithmic scaling).

It can be observed that, for the data set $\Dvalid$, the method reconstructs most of the images exactly whenever $\Delta  \geq 0.07$ for $\fcut=12 $ and whenever $\Delta  \geq 0.04$ for $\fcut=18$. Regarding $\Dinvalid$, it is remarkable that the reconstruction performance is clearly worse, even though only two single values were flipped for each image in order to violate \cref{ass:consistentGradient}. In particular, while for $\fcut=18$ exact reconstruction is possible for $\Delta \geq 0.08$, in case of $\fcut=12$ a consistent exact reconstruction is not obtained for the considered jump-point distances.
This strongly indicates that \cref{ass:consistentGradient} (or at least some consistency condition on the gradient directions) indeed makes a difference for exact reconstruction at a low frequency cutoff.

\paragraph*{Experiments regarding convergence for vanishing noise.} 
In order to numerically evaluate the convergence result, a range of noise variances $\delta_i = 10^{k_i}$ with $k_i$, $i=1,\ldots,20$ equally spaced between $-3$ and $3$, was defined, and for each image $u^\dagger$ in $\Dconv$, Gaussian noise
with zero mean and a variance such that 
\[ \frac{1}{2}\| f^{\delta_i} - f^\dagger \|_2 ^2 = \delta_i\] in expectation
was added to $f^\dagger = Ku^\dagger$, where $f^{\delta_i} $ is $f^\dagger$ with the added noise. This resulted in a total of $46\cdot 20 = 920 $ data points $f^{\delta}$ corresponding to all combinations of images in $\Dconv$ and noise variances $\delta_i$. For each experiment with a noise variance $\delta_i$, the constant $\alpha$ was chosen as $C\sqrt{\delta_i}$, where $C = 1/0.028$ was determined beforehand using a grid-search on a small set of test images and noise variances.

The results, showing the average $L^1$ error over all reconstructed images with fixed $\delta_i$, can be found the double-logarithmic plot of \cref{fig:error_plots}. As reference, the figure provides also plots of $\delta^{1/2}$ and $\delta^{1/4}$. It can be observed that, different from the analytic result of \cref{thm:convergence_intro}, the convergence rate closely matches the rate $\delta^{1/2}$. This confirms that the rate $\delta^{1/4}$ of \cref{thm:convergence_intro} is not optimal and can possibly be improved to a rate $\delta^{1/2}$ based on more adapted source conditions (which is future work), see \cref{rem:rate_not_optimal} for a discussion.

\paragraph*{Experiments on a real test image.} 
At last, we provide an experiment that visualizes the effect of anisotropic TV regularization for varying frequency cutoffs on a real test image, see \cref{fig:cameraman}. Here, the data was generated by applying a truncated Fourier transform to the \emph{Cameraman} test image with frequency cutoffs $\fcut = 3,6,9,12,15,18,21,24,27,30,33$. Reconstruction was carried out by solving \ref{eq:main_min_prob} with $\alpha = 0$, i.e., constraining those Fourier frequencies where data is available to match the data exactly.

While the reconstructions, in particular for low frequency cutoffs, suffer from block artifacts in the form of rectangles with constant gray values, also diagonal and curved edges are visible. This indicates numerically that block-images satisfying assumption \ref{ass:consistentGradient} are not the only class of images where exact reconstruction with anisotropic TV may be possible. Furthermore, in view of results from \cite{Boyer19representer_mh,Carioni20_representer}, it provides hints on possible extremal points of the anisotropic TV regularizer, in particular that they apparently do not (only) consist of rectangles.

\begin{figure}[t]
	\centering
	\begin{subfigure}[b]{0.49\textwidth}
	\centering
	\includegraphics[scale=0.37]{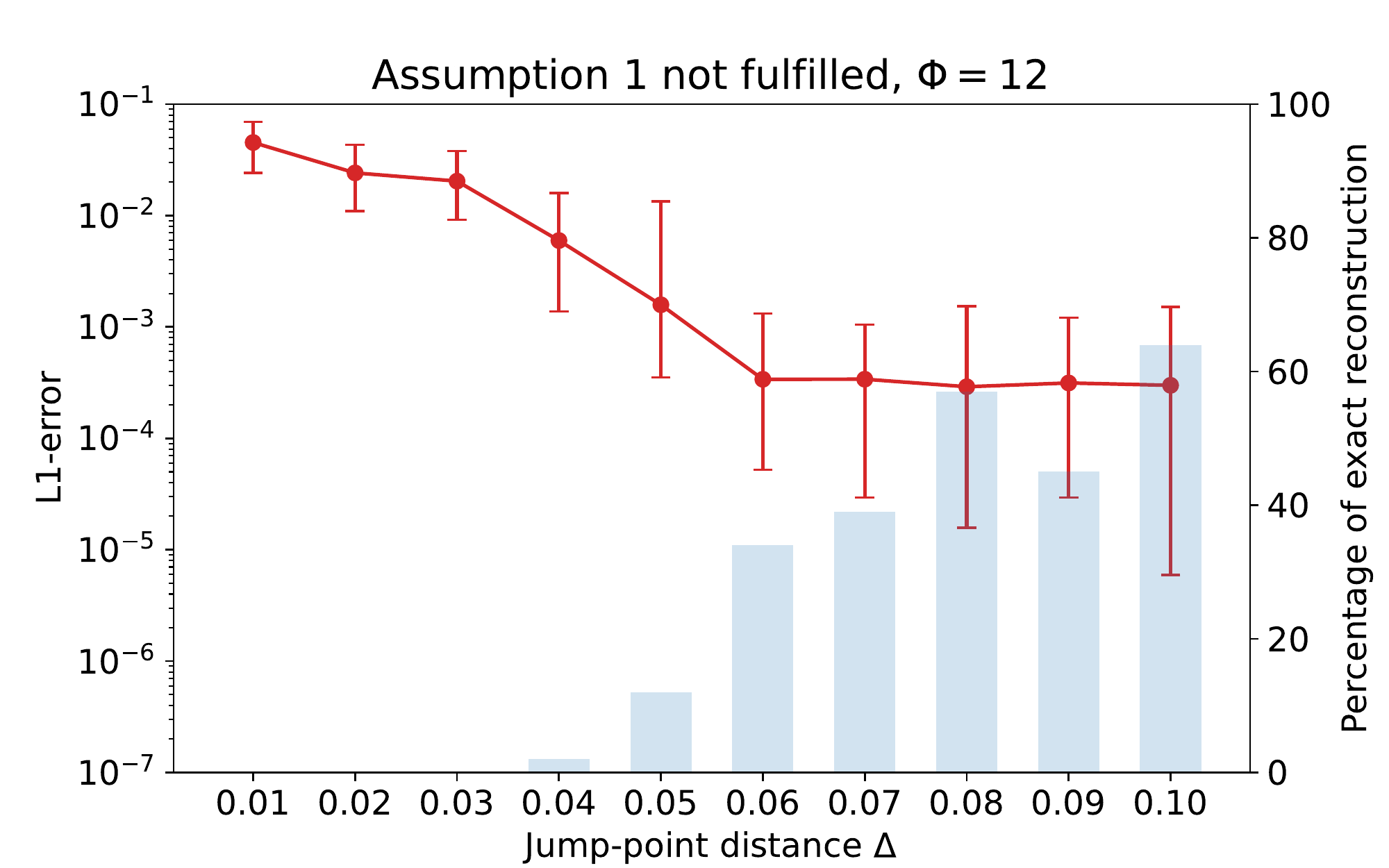}
	\end{subfigure}
	\begin{subfigure}[b]{0.49\textwidth}
	\centering
	\includegraphics[scale=0.37]{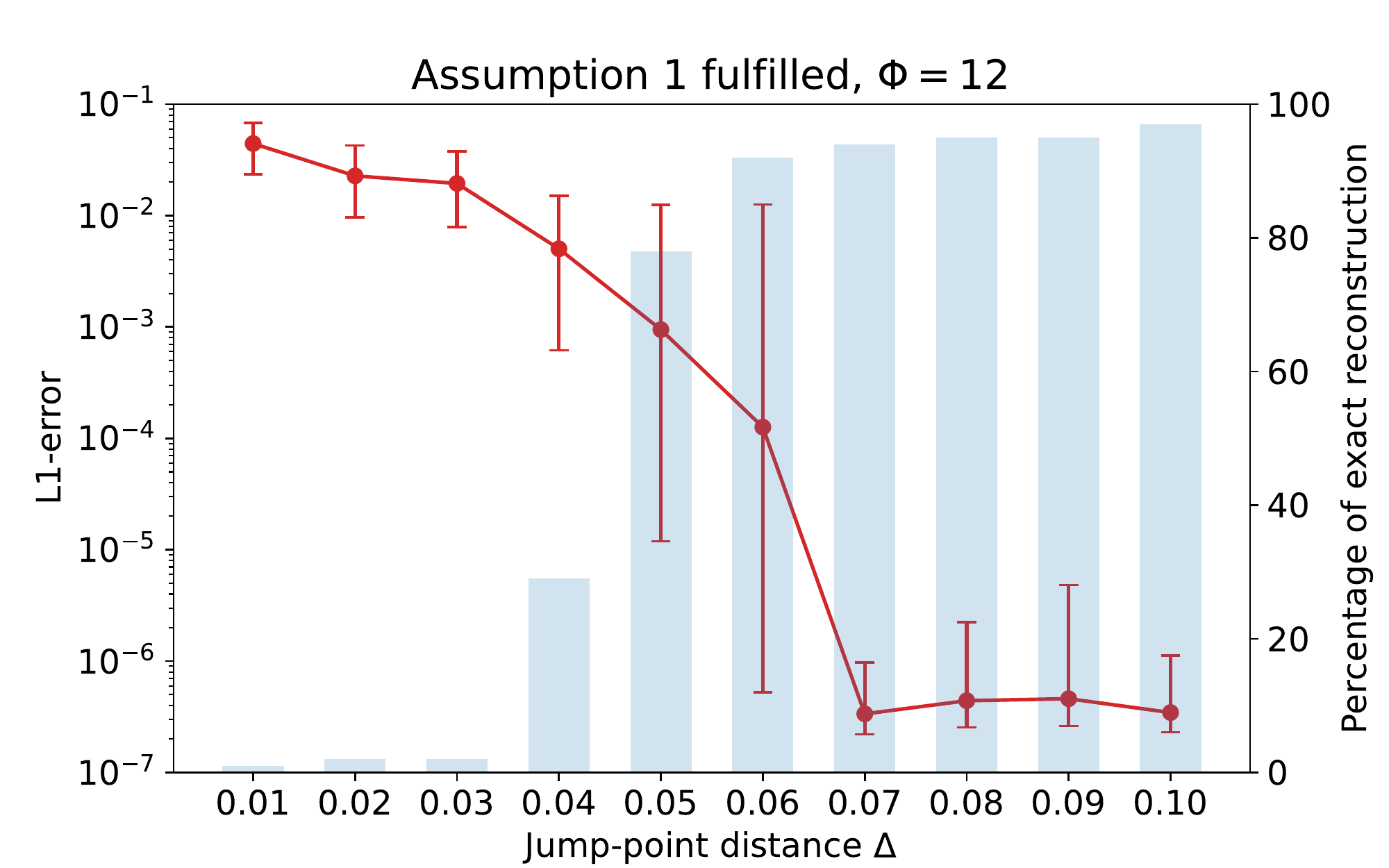}
	\end{subfigure}
	\begin{subfigure}[b]{0.49\textwidth}
	\centering
	\includegraphics[scale=0.37]{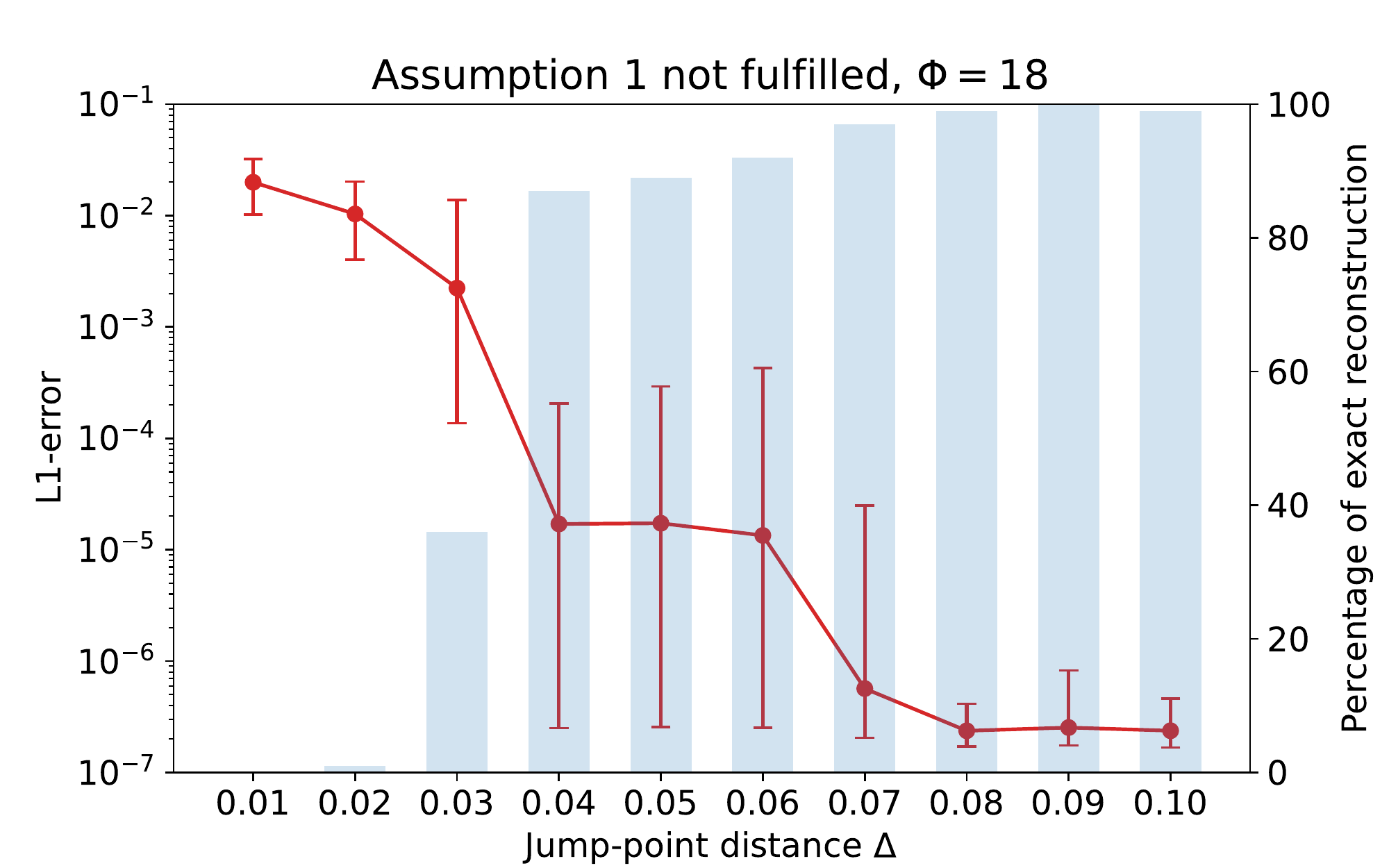}
	\end{subfigure}
	\begin{subfigure}[b]{0.49\textwidth}
	\centering
	\includegraphics[scale=0.37]{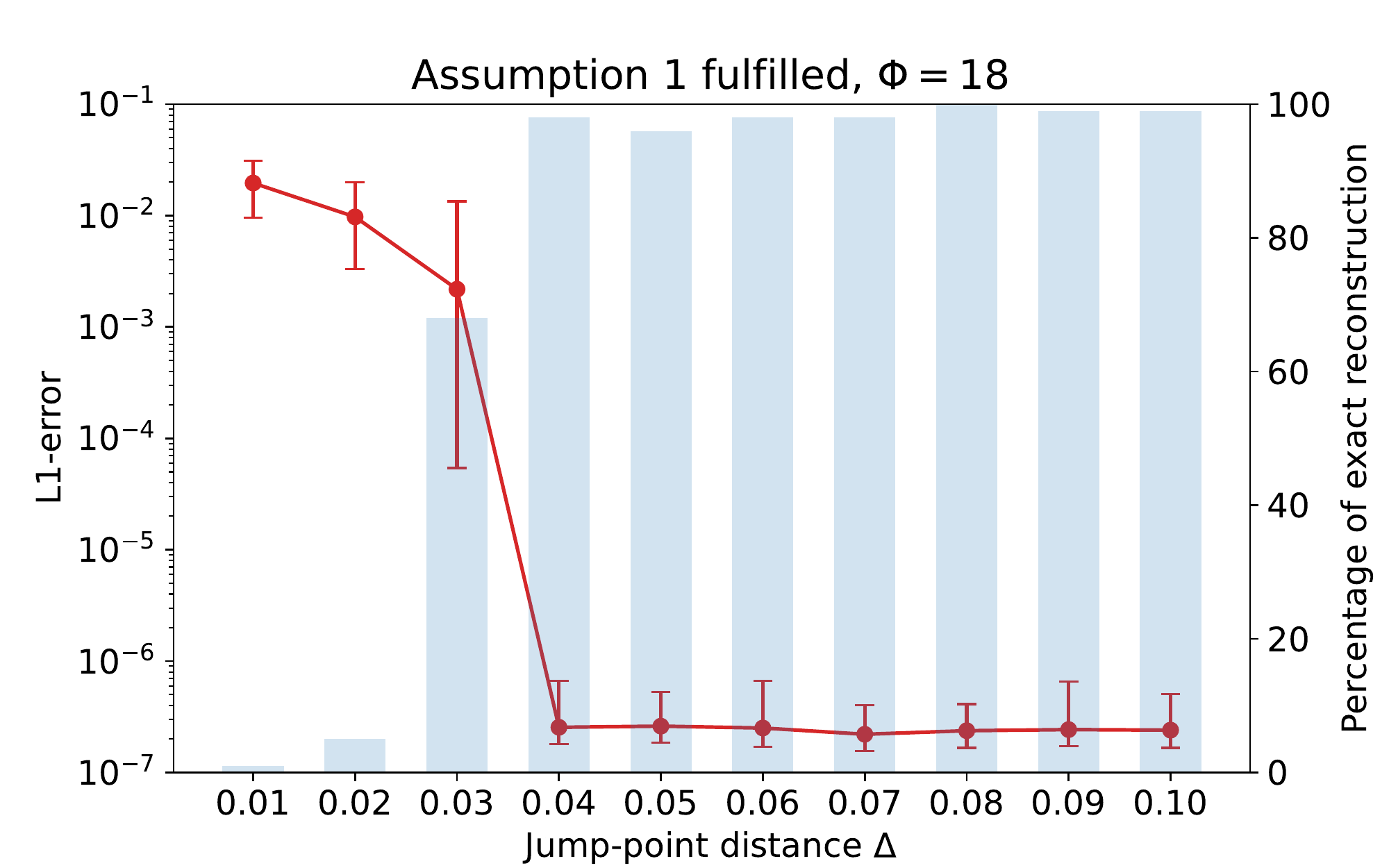}
	\end{subfigure}
\caption{\label{fig:error_plots} $L^1$ error (red line) and percentage of exact reconstruction (blue bars) for two different frequency cutoffs (top $\fcut=12$ and bottom $\fcut=18$) and in dependence on the gradient condition \cref{ass:consistentGradient} being fulfilled (right) or not (left).}
\end{figure}

\begin{figure}
\centering
\includegraphics[scale=0.4]{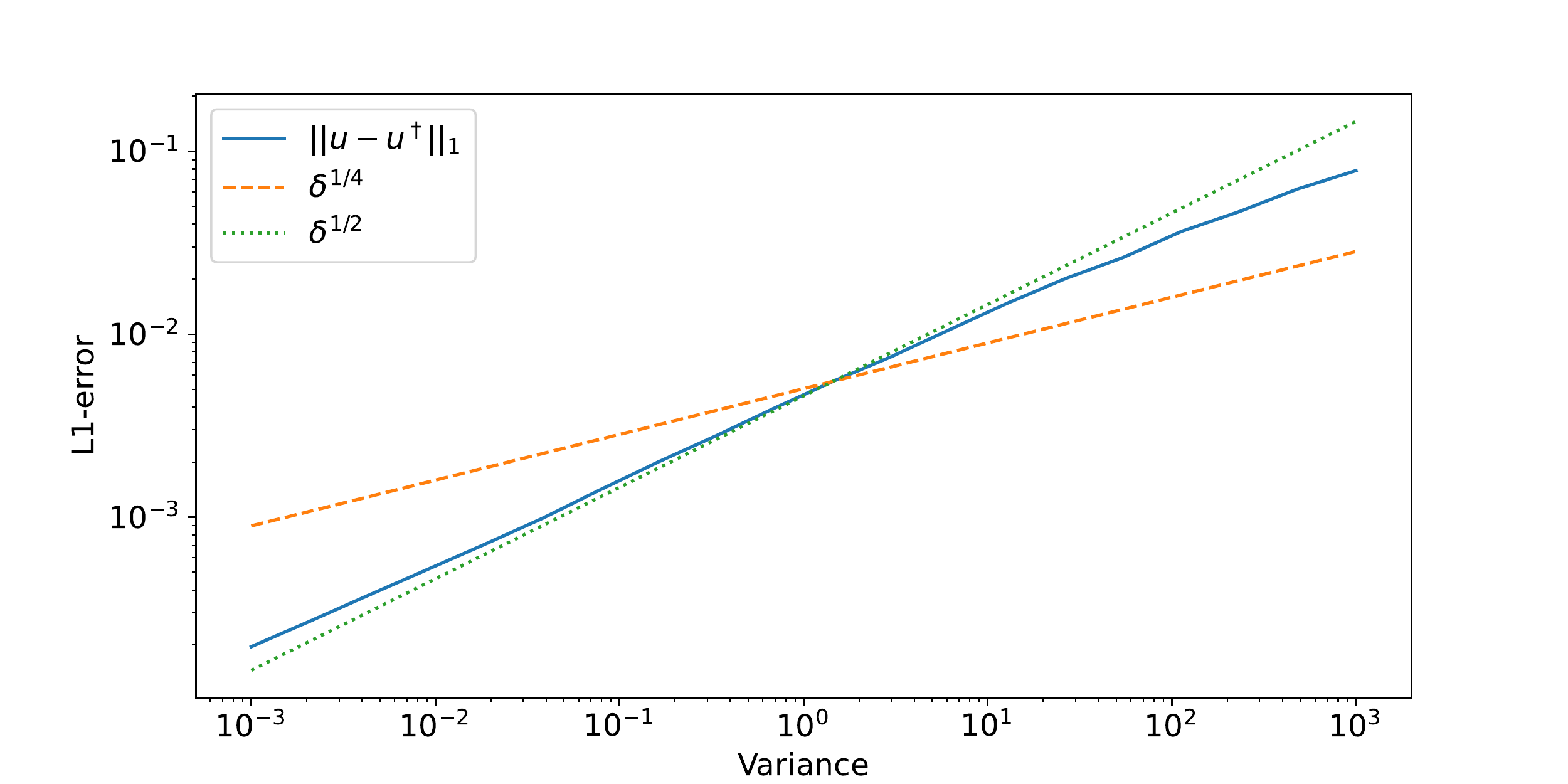}
\caption{\label{fig:rate_plot} Double-logarithmic plot of the $L^1$-error of reconstructions in dependence on the noise variance $\delta$ (blue) together with the plots of $\delta^{1/4}$ (orange, dashed) and $\delta^{1/2}$ (green, dotted) as reference.} 
\end{figure}

\begin{figure}[t]
	\centering
	
	\includegraphics[width=0.16\linewidth]{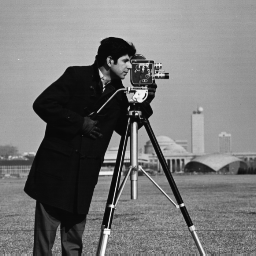}
	\includegraphics[width=0.16\linewidth]{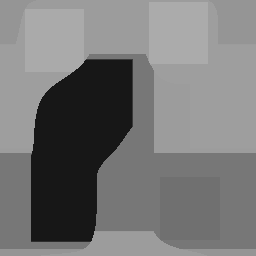}
	\includegraphics[width=0.16\linewidth]{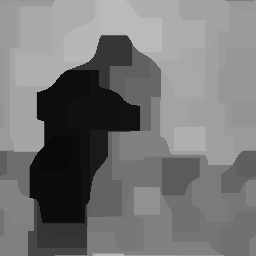}
	\includegraphics[width=0.16\linewidth]{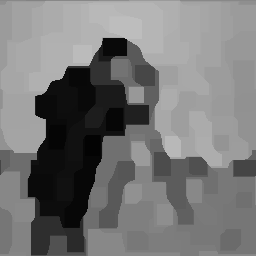}
	\includegraphics[width=0.16\linewidth]{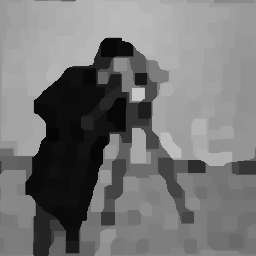}
	\includegraphics[width=0.16\linewidth]{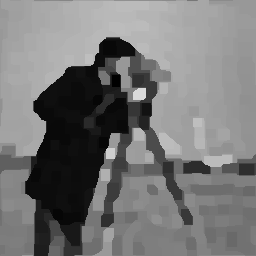}
	
\vspace*{0.1cm}

	\includegraphics[width=0.16\linewidth]{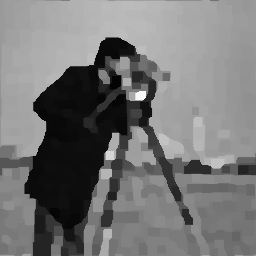}
	\includegraphics[width=0.16\linewidth]{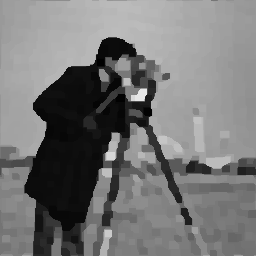}
	\includegraphics[width=0.16\linewidth]{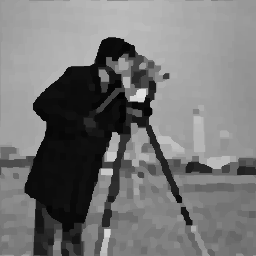}
	\includegraphics[width=0.16\linewidth]{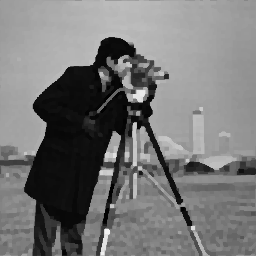}
	\includegraphics[width=0.16\linewidth]{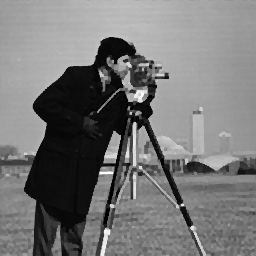}
	\includegraphics[width=0.16\linewidth]{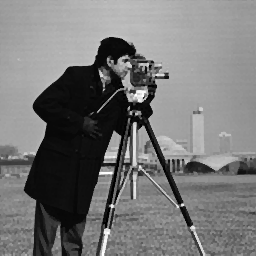}
\caption{ \label{fig:cameraman}Reconstructions of the \emph{Cameraman} image for different values of the cutoff frequency $\fcut$. From left to right, top to bottom: Original, $\fcut = 3,6,9,12,15,18,21,24,27,30,33$.}
\end{figure}

\section*{Acknowledgements}
BW's work was supported by the Deutsche Forschungsgemeinschaft (DFG, German Research Foundation)
under Germany's Excellence Strategy -- EXC 2044 --, Mathematics M\"unster: Dynamics -- Geometry -- Structure,
and under the Collaborative Research Centre 1450--431460824, InSight, University of M\"unster.

\bibliography{lit_dat}

\begin{thebibliography}{10}

\bibitem{AlbertiAmmariRomeroWintz2019}
G.~S. Alberti, H.~Ammari, F.~Romero, and T.~Wintz.
\newblock Dynamic spike superresolution and applications to ultrafast
  ultrasound imaging.
\newblock {\em SIAM Journal on Imaging Sciences}, 12(3):1501--1527, 2019.

\bibitem{Ambrosio}
L.~Ambrosio, N.~Fusco, and D.~Pallara.
\newblock {\em Functions of bounded variation and free discontinuity problems}.
\newblock Oxford University Press, 2000.

\bibitem{Boyer19representer_mh}
C.~Boyer, A.~Chambolle, Y.~D. Castro, V.~Duval, F.~De~Gournay, and P.~Weiss.
\newblock On representer theorems and convex regularization.
\newblock {\em SIAM Journal on Optimization}, 29(2):1260--1281, 2019.

\bibitem{Carioni20_representer}
K.~Bredies and M.~Carioni.
\newblock Sparsity of solutions for variational inverse problems with
  finite-dimensional data.
\newblock {\em Calculus of Variations and Partial Differential Equations},
  59(1):1--26, 2020.

\bibitem{holler16tvsubdif_mh}
K.~Bredies and M.~Holler.
\newblock A pointwise characterization of the subdifferential of the total
  variation functional.
\newblock arXiv:1609.08918, 2016.

\bibitem{BrBrTh97}
A.~M. Bruckner, J.~B. Bruckner, and B.~S. Thomson.
\newblock {\em Real analysis}.
\newblock Prentice-Hall (Pearson), Upper Saddle River, N.J., 1997.
\newblock Reprint of the 1962 original.

\bibitem{BurgerOsher2004}
M.~Burger and S.~Osher.
\newblock Convergence rates of convex variational regularization.
\newblock {\em Inverse Problems}, 20(5):1411--1421, 2004.

\bibitem{Candes2013}
E.~J. Cand\`es and C.~Fernandez-Granda.
\newblock Super-resolution from noisy data.
\newblock {\em J. Fourier Anal. Appl.}, 19(6):1229--1254, 2013.

\bibitem{Candes2014}
E.~J. Candès and C.~Fernandez-Granda.
\newblock Towards a mathematical theory of super-resolution.
\newblock {\em Communications on Pure and Applied Mathematics}, 67(6):906--956,
  2014.

\bibitem{Pock11_primal_dual}
A.~Chambolle and T.~Pock.
\newblock A first-order primal-dual algorithm for convex problems with
  applications to imaging.
\newblock {\em Journal of Mathematical Imaging and Vision}, 40:120--145, 2011.

\bibitem{Castro2012}
Y.~de~Castro and F.~Gamboa.
\newblock Exact reconstruction using beurling minimal extrapolation.
\newblock {\em Journal of Mathematical Analysis and Applications},
  395(1):336--354, 2012.

\bibitem{DenoyelleDuvalPeyreSoubies2020}
Q.~Denoyelle, V.~Duval, G.~Peyr\'{e}, and E.~Soubies.
\newblock The sliding {F}rank-{W}olfe algorithm and its application to
  super-resolution microscopy.
\newblock {\em Inverse Problems}, 36(1):014001, 42, 2020.

\bibitem{Duval2015}
V.~Duval and G.~Peyré.
\newblock Exact support recovery for sparse spikes deconvolution.
\newblock {\em Foundations of Computational Mathematics}, 15(5):1315--1355,
  2015.

\bibitem{HollerSchlueterWirth2022}
M.~{Holler}, A.~{Schl{\"u}ter}, and B.~{Wirth}.
\newblock Dimension reduction, exact recovery, and error estimates for sparse
  reconstruction in phase space.
\newblock arXiv:2112.09743, 2021.

\bibitem{paper_source_code}
M.~Holler and B.~Wirth.
\newblock Source code to reproduce the results of "{Exact reconstruction and
  reconstruction from noisy data with anisotropic total variation}".
\newblock \url{https://github.com/hollerm/ani_tv_recovery}, 2023.

\bibitem{klockner2012pycuda}
A.~Kl{\"o}ckner, N.~Pinto, Y.~Lee, B.~Catanzaro, P.~Ivanov, and A.~Fasih.
\newblock {PyCUDA} and {PyOpenCL}: A scripting-based approach to {GPU} run-time
  code generation.
\newblock {\em Parallel Computing}, 38(3):157--174, 2012.

\bibitem{Plonka2013prony_spline}
G.~Plonka and M.~Wischerhoff.
\newblock How many fourier samples are needed for real function reconstruction?
\newblock {\em Journal of Applied Mathematics and Computing}, 42(1):117--137,
  2013.

\end{thebibliography}
\bibliographystyle{abbrv}

\end{document}